\documentclass[11pt,a4paper]{article}

\pdfoutput=1

\usepackage[T1]{fontenc}
\usepackage[utf8]{inputenc}
\usepackage[english]{babel}
\usepackage{amsmath,amssymb,amsthm}
\usepackage[margin=2.6cm]{geometry}
\usepackage{array}
\usepackage{booktabs}
\usepackage{enumitem}
\usepackage[hidelinks]{hyperref}
\usepackage[expansion=false]{microtype}

\theoremstyle{plain}
\newtheorem{theorem}{Theorem}[section]
\newtheorem{proposition}[theorem]{Proposition}
\newtheorem{lemma}[theorem]{Lemma}
\newtheorem{corollary}[theorem]{Corollary}
\theoremstyle{definition}
\newtheorem{definition}[theorem]{Definition}
\newtheorem{example}[theorem]{Example}
\theoremstyle{remark}
\newtheorem{remark}[theorem]{Remark}

\newcommand{\R}{\mathbb{R}}
\newcommand{\norm}[1]{\left\lVert #1\right\rVert}
\newcommand{\abs}[1]{\left\lvert #1\right\rvert}
\newcommand{\LL}{L^{2}(\Omega)}
\newcommand{\HH}{H_{0}^{1}(\Omega)}
\newcommand{\HM}{H^{-1}(\Omega)}
\newcommand{\LLtwo}{L^{2}\bigl((0,T);L^{2}(\Omega)\bigr)}
\newcommand{\LLr}{L^{r}\bigl((0,T);L^{2}(\Omega)\bigr)}
\newcommand{\LLq}{L^{q}\bigl((0,T);L^{2}(\Omega)\bigr)}
\newcommand{\LHone}{L^{2}\bigl((0,T);H_{0}^{1}(\Omega)\bigr)}
\newcommand{\LHmone}{L^{2}\bigl((0,T);H^{-1}(\Omega)\bigr)}
\newcommand{\LinfL}{L^{\infty}\bigl((0,T);L^{2}(\Omega)\bigr)}
\newcommand{\Wscale}{W^{\alpha}_{2,2}\bigl((0,T);H_{0}^{1}(\Omega),H^{-1}(\Omega)\bigr)}
\newcommand{\WscaleG}{W^{\alpha}_{2,2}(g)}

\newcommand{\IL}[1]{I^{#1}_{0+}}            
\newcommand{\IR}[1]{I^{#1}_{T-}}            
\newcommand{\DL}{D^{\alpha}_{0+}}           
\newcommand{\RLD}{D^{\alpha}_{T-}}          
\newcommand{\CD}{{}^{C}\!D^{\alpha}_{0+}}   

\newcommand{\Hone}{\textup{(H1)}}
\newcommand{\Htwo}{\textup{(H2)}}
\newcommand{\Hthree}{\textup{(H3)}}
\newcommand{\Hthreep}{\textup{(H3$'$)}}
\newcommand{\Hfour}{\textup{(H4)}}

\numberwithin{equation}{section}

\begin{document}

\title{\bfseries Source integrability and the frontier $\alpha=\tfrac12$\\
in time-fractional diffusion}

\author{Edison A. V\'asquez-Mart\'inez\thanks{Instituto de Matem\'aticas, Universidad de
Antioquia, Medell\'in, Colombia. \texttt{edison.vasquez@udea.edu.co}.
ORCID: 0009-0003-6659-1122.}
\qquad
Mauricio A. Londo\~no-Arboleda\thanks{Instituto de Matem\'aticas, Universidad de Antioquia,
Medell\'in, Colombia. \texttt{alejandro.londono@udea.edu.co}.
ORCID: 0000-0002-2913-9044.}}

\date{October 2026}

\maketitle

\begin{abstract}
Let $\Omega\subset\R^{d}$ be a bounded Lipschitz domain, let $\alpha\in(0,1)$, and let
$\partial_{t}^{\alpha}$ denote the Caputo derivative. For
$\partial_{t}^{\alpha}u-\nabla\cdot\bigl(A(x)\nabla u\bigr)=f$ with homogeneous Dirichlet
condition and initial datum $g\in L^{2}(\Omega)$, we identify which of the thresholds appearing
in the energy method belong to the method and which to the equation. We compare
$f\in L^{2}\bigl((0,T);L^{2}(\Omega)\bigr)$ and $f\in L^{r}\bigl((0,T);L^{2}(\Omega)\bigr)$ with
$r>2/\alpha$; in the subdiffusive range the second implies the first and not conversely.
Existence is not at issue and we claim none of it: it is known, for both hypotheses and under
weaker ones, from the work of Zacher. What we determine is the exact set of exponents each route
admits and what the exponents left outside it cost. Under the weaker hypothesis the construction
below reaches $\alpha\in(\frac12,1)$, under the stronger one every $\alpha\in(0,1)$, and in both
cases the solution lies in
$L^{\infty}\bigl((0,T);L^{2}(\Omega)\bigr)$, by the mode representation and not by the energy
route, which yields no such bound under the weaker hypothesis. The restriction $\alpha>\frac12$
comes from the compactness criterion, whose admissible exponents form $[1,2\alpha)$; the exponent
$r>2/\alpha$ is exactly what the Gr\"onwall step consumes. A concentrating family shows that
nothing weaker suffices along that route, and the mode representation lowers the requirement to
$1/\alpha$. The frontier $\alpha=\frac12$ belongs to the equation: for every
$\alpha\in(0,\frac12]$ there is $f\in L^{2}\bigl((0,T);L^{2}(\Omega)\bigr)$ whose weak solution
leaves $L^{\infty}\bigl((0,T);L^{2}(\Omega)\bigr)$. Uniqueness and Hadamard well-posedness
follow.
\end{abstract}



\section{The problem}\label{sec:problem}

\subsection{The equation}

Let $T>0$ be a fixed horizon. For
$\alpha\in(0,1)$ we study the time-fractional diffusion problem
\begin{equation}\label{eq:P}
(P_{\alpha})\qquad
\begin{cases}
\partial_{t}^{\alpha}u-\nabla\cdot\bigl(A(x)\nabla u\bigr)=f, & (x,t)\in\Omega\times(0,T),\\[2pt]
u=0, & (x,t)\in\partial\Omega\times[0,T],\\[2pt]
u(\cdot,0)=g, & \text{in }\Omega,
\end{cases}
\end{equation}
where $\partial^{\alpha}_{t}$ is the Caputo derivative of Definition~\ref{def:caputo} below, and
under the following hypotheses:

\begin{itemize}[leftmargin=3.4em]
\item[\Hone] $\Omega\subset\R^{d}$ is a bounded domain with Lipschitz boundary.
\item[\Htwo] $A:\Omega\to\R^{d\times d}$ is measurable, symmetric for a.e. $x$, bounded and
uniformly elliptic: there are constants $a_{1}\ge a_{0}>0$ such that, for a.e. $x\in\Omega$ and
every $\xi\in\R^{d}$,
\[
A(x)\xi\cdot\xi\ge a_{0}\abs{\xi}^{2},
\qquad
\abs{A(x)\xi}\le a_{1}\abs{\xi} .
\]
\end{itemize}

We impose exactly one initial condition, as the order requires: it is the datum that appears in
the fundamental theorem of fractional calculus. Equations of this type model anomalous transport
in disordered and porous media, where the mean-square displacement grows like $t^{\alpha}$
instead of linearly. We fix the order once and for all in $(0,1)$, and we state and claim
nothing outside that range.

Existence and uniqueness for \eqref{eq:P} are known. Zacher \cite{zacher09} proved both in an
abstract Hilbert-space formulation, for every $\alpha\in(0,1)$, with merely bounded measurable
coefficients and under the weaker hypothesis $f\in L^{2}\bigl((0,T);\HM\bigr)$, see
\cite{zacher19}. Kubica and Yamamoto \cite{kubica} recast the theory in
the scale $W^{\alpha}_{2,2}$ used here and extended it to time-dependent coefficients, and
Sakamoto and Yamamoto \cite{sakamoto} treated the problem through the spectral representation.
Compactness in this setting rests on the criteria of Li and Liu \cite{liliu}, which replace the
Aubin--Lions--Simon lemma \cite{simon,lions} when the time derivative has order $\alpha<1$.
Regularity of a different kind is studied by Loreti and Sforza \cite{loreti}, and smoothing
estimates under pointwise weighted bounds on the source by McLean, Mustapha, Ali and Knio
\cite{mclean}.

Several recent works address the question of how integrable the source has to be. Kubica,
Ryszewska and Zacher \cite{kubicaRZ} prove a weak Harnack inequality and interior H\"older
continuity for $\partial_{t}(k*(u-u_{0}))-\nabla\cdot(A(t,x)\nabla u)=f$ with a general kernel,
admitting $f$ in a mixed Lebesgue space $L^{q_{1}}\bigl((0,T);L^{q_{2}}(\Omega)\bigr)$. Floridia,
G\"olgeleyen and Yamamoto \cite{floridia} construct strong solutions in Banach spaces for
$F\in L^{q}(0,T;X)$. In the numerical analysis of the same equation the integrability of the
source in time governs the attainable order of convergence, and sources singular in time are
treated separately there \cite{zhoutian,shichen}. None settles the questions we answer below, and it is worth
stating at the outset what those are, so that the nature of the contribution is not mistaken.
Existence under either hypothesis is known, and none of it is claimed here. What we establish is,
first, the exact window $[1,2\alpha)$ of exponents with which the compactness criterion can be
applied under \Hthree, both halves of that identification being proved; second, that $2/\alpha$ is
precisely the integrability the energy route consumes, while the conclusion it is used to reach
requires only $1/\alpha$; and third, that the threshold $\alpha=\frac12$ which follows is no
artefact of either argument, since for every $\alpha\in(0,\frac12]$ there is a source satisfying
\Hthree{} whose solution leaves $\LinfL$.

In addition, we will consider the following hypotheses for the source term and the initial
condition:

\begin{itemize}[leftmargin=3.4em]
\item[\Hthree] $f\in\LLtwo$.
\item[\Hthreep] $f\in\LLr$ for some $r>2/\alpha$.
\item[\Hfour] $g\in\LL$.
\end{itemize}

In all of the works just cited the source is taken in $\LLtwo$, or in $L^{2}\bigl((0,T);\HM\bigr)$,
or subject to a weighted pointwise bound, and the estimates the argument needs are then read off
from that choice. The direct route, which would bound the fractional average of the source under
\Hthree{} alone, fails for the reason Proposition~\ref{prop:nobound} makes precise.

\subsection{Plan of the paper}

The paper is organised as follows. Section~\ref{sec:hyp} comments on the hypotheses stated above
and Section~\ref{sec:relation} orders the two on the source. Section~\ref{sec:common} collects the common development: Galerkin scheme, energy inequality,
compactness, uniqueness, and the spectral bound on the approximations.
Sections~\ref{sec:caseA} and~\ref{sec:caseB} prove the two theorems,
Section~\ref{sec:comparison} compares them, and Section~\ref{sec:sharp} settles which thresholds
belong to the problem. Section~\ref{sec:examples} contains one example and
Section~\ref{sec:conclusions} the conclusions. The remaining subsections of the present section
fix notation, the fractional operators, and the notion of weak solution.

We treat \Hthree{} first because it is the weaker hypothesis, so that the order exhibits what
can be obtained from the weaker assumption before any of the cost of strengthening it is
incurred.

\subsection{Notation}

We write $(\cdot,\cdot)_{\LL}$ for the inner product of $\LL$ and $\langle\cdot,\cdot\rangle$ for
the $\HM$--$\HH$ duality, and set
\begin{equation}\label{eq:bilinear}
a(u,v):=\int_{\Omega}A(x)\nabla u\cdot\nabla v\,dx,\qquad u,v\in\HH .
\end{equation}
On $\HH$ we use $\norm{u}_{\HH}=\norm{\nabla u}_{\LL}$, equivalent to the usual norm by
Poincar\'e's inequality, whose constant is written $C_{P}$, so that
$\norm{u}_{\LL}\le C_{P}\norm{u}_{\HH}$. Under \Hone{} and \Htwo{} the form
$a$ is coercive and bounded,
\begin{equation}\label{eq:coercive}
a(u,u)\ge a_{0}\norm{u}^{2}_{\HH},
\qquad
\abs{a(u,v)}\le a_{1}\norm{u}_{\HH}\norm{v}_{\HH},
\end{equation}
where $a_{0}$ is the ellipticity constant of \Htwo{} and $a_{1}$ the bound of the coefficient.
The form is symmetric, $a(u,v)=a(v,u)$, because $A(x)$ is; symmetry enters only once, in
Lemma~\ref{lem:spectral}, but there it is decisive.

We always write $\alpha$ for the order of the equation. The letters $\gamma$ and $\delta$ denote
no quantity of the problem: they are the dummy variables in which we state the auxiliary results
and which we instantiate later with concrete values. We write $\varepsilon$ for the free
parameter of Young's inequality in \eqref{eq:diffineq}, which the Gr\"onwall step carries into the
stability constant of Theorem~\ref{thm:stability}, and $\varrho$ for the radius of the
concentrating families of Sections~\ref{sec:source} and~\ref{sec:windowA}; the two are kept
apart. Finally, we write $C$ for a
constant that may change from line to line and that depends only on $T$, $\alpha$, $a_{0}$,
$a_{1}$, $C_{P}$ and the norms of the data, never on the index $n$ of the Galerkin approximation.

Fractional operators are always written with their argument and their point of evaluation
displayed, as in $\bigl(\IL{\gamma}\psi\bigr)(t)$. The two operators of
Section~\ref{sec:operators} differ only in the endpoint over which they integrate, and
suppressing either the argument or the point would make them indistinguishable.

\subsection{The fractional operators}\label{sec:operators}

We use the classical definitions of Kilbas, Srivastava and Trujillo \cite{kilbas}, specialised to
the finite interval $[0,T]$ and to real orders, which is all we need, and we keep their
subscripts $0+$ and $T-$.

\begin{definition}[Riemann--Liouville fractional integrals]\label{def:RL}
Let $\gamma>0$ and let $\psi\in L^{1}(0,T)$. The left- and right-sided Riemann--Liouville
fractional integrals of order $\gamma$ are
\begin{align}
\bigl(\IL{\gamma}\psi\bigr)(t)&:=\frac{1}{\Gamma(\gamma)}\int_{0}^{t}
\frac{\psi(s)}{(t-s)^{1-\gamma}}\,ds,\qquad\text{a.e. }t\in(0,T],\label{eq:RLleft}\\[2pt]
\bigl(\IR{\gamma}\psi\bigr)(t)&:=\frac{1}{\Gamma(\gamma)}\int_{t}^{T}
\frac{\psi(s)}{(s-t)^{1-\gamma}}\,ds,\qquad\text{a.e. }t\in[0,T),\label{eq:RLright}
\end{align}
which are \cite[Eqs.~(2.1.1) and~(2.1.2)]{kilbas}, and we set
$\IL{0}:=\IR{0}:=\mathrm{Id}$. Both extend componentwise to functions with values in a Banach
space. Only $\IL{\gamma}$ acts on the solution; $\IR{\gamma}$ acts on test functions and appears
solely in \eqref{eq:intbyparts} and in the terms that identity produces.
\end{definition}

\begin{lemma}[Properties of the fractional integrals]\label{lem:RLprops}
Let $\gamma,\gamma'>0$ and let $p\in[1,\infty]$. Then
\begin{enumerate}[label=\textup{(\roman*)},leftmargin=2.4em]
\item $\IL{\gamma}$ and $\IR{\gamma}$ map $L^{p}(0,T)$ into itself, and
\begin{equation}\label{eq:RLbounded}
\norm{\IL{\gamma}\psi}_{L^{p}(0,T)}\le\frac{T^{\gamma}}{\Gamma(\gamma+1)}
\norm{\psi}_{L^{p}(0,T)} ;
\end{equation}
\item $\IL{\gamma}\IL{\gamma'}=\IL{\gamma+\gamma'}$ and
$\IR{\gamma}\IR{\gamma'}=\IR{\gamma+\gamma'}$ on $L^{p}(0,T)$;
\item $\IL{\gamma}$ is linear and order preserving, and
$\bigl(\IL{\gamma}\psi\bigr)(t)$ depends only on $\psi|_{(0,t)}$; in particular $\IL{\gamma}\psi$
is unchanged on $(0,T_{1})$ if $\psi$ is modified on $(T_{1},T)$;
\item if $\varphi\in L^{\infty}(0,T)$ and $\psi\in L^{1}(0,T)$, then
\begin{equation}\label{eq:intbypartsI}
\int_{0}^{T}\varphi(t)\,\bigl(\IL{\gamma}\psi\bigr)(t)\,dt
=\int_{0}^{T}\psi(t)\,\bigl(\IR{\gamma}\varphi\bigr)(t)\,dt .
\end{equation}
\end{enumerate}
\end{lemma}

\begin{proof}
Parts (i), (ii) and (iv) are \cite[Lemmas~2.1, 2.3 and~2.7]{kilbas}; see also
\cite[Ch.~1, \S\S2--3]{samko}, \cite[Sec.~2.3]{podlubny}, \cite[Ch.~2]{gripenberg}. Part (iii) is
immediate from \eqref{eq:RLleft}, the integration running over $(0,t)$ and over nothing else.
\end{proof}

\begin{definition}[Riemann--Liouville fractional derivatives]\label{def:RLder}
Let $\alpha\in(0,1)$. For $y$ with $\IL{1-\alpha}y\in AC([0,T])$, and for $\phi$ with
$\IR{1-\alpha}\phi\in AC([0,T])$,
\begin{align}
\bigl(\DL y\bigr)(t)&:=\frac{d}{dt}\bigl(\IL{1-\alpha}y\bigr)(t)
=\frac{1}{\Gamma(1-\alpha)}\frac{d}{dt}\int_{0}^{t}\frac{y(s)}{(t-s)^{\alpha}}\,ds,
\label{eq:RLderleft}\\[2pt]
\bigl(\RLD\phi\bigr)(t)&:=-\frac{d}{dt}\bigl(\IR{1-\alpha}\phi\bigr)(t)
=-\frac{1}{\Gamma(1-\alpha)}\frac{d}{dt}\int_{t}^{T}\frac{\phi(s)}{(s-t)^{\alpha}}\,ds ,
\label{eq:RLderright}
\end{align}
the operators of \cite[Sec.~2.1]{kilbas}. The definitions extend componentwise to Banach-space
valued functions.
\end{definition}

\begin{definition}[Caputo derivative]\label{def:caputo}
Let $X$ be a Banach space, $\alpha\in(0,1)$, and let $u:[0,T]\to X$ be such that $u(0)$ is
defined. Following \cite[Eqs.~(2.4.1)--(2.4.5)]{kilbas} we set
\begin{equation}\label{eq:caputodef}
\partial^{\alpha}_{t}u:=\DL\bigl[u-u(0)\bigr],
\end{equation}
which is the Caputo derivative $\CD u$ of \cite{kilbas}; we keep the shorter symbol
$\partial^{\alpha}_{t}$ throughout.
\end{definition}

\begin{lemma}[Identities for the Caputo derivative]\label{lem:caputoprops}
Let $\alpha\in(0,1)$ and $u\in AC\bigl([0,T];X\bigr)$. Then $\partial^{\alpha}_{t}u$ exists a.e.
and
\begin{align}
\bigl(\partial^{\alpha}_{t}u\bigr)(t)&=\bigl(\IL{1-\alpha}u'\bigr)(t)
=\frac{1}{\Gamma(1-\alpha)}\int_{0}^{t}\frac{u'(s)}{(t-s)^{\alpha}}\,ds ,\label{eq:caputo}\\[2pt]
\bigl(\DL u\bigr)(t)&=\bigl(\partial^{\alpha}_{t}u\bigr)(t)
+\frac{t^{-\alpha}}{\Gamma(1-\alpha)}\,u(0),\label{eq:RLvsC}\\[2pt]
\bigl(\IL{\alpha}\partial^{\alpha}_{t}u\bigr)(t)&=u(t)-u(0),\qquad t\in[0,T].\label{eq:ftfc}
\end{align}
\end{lemma}

\begin{proof}
Identity \eqref{eq:caputo} is \cite[Eqs.~(2.4.1)--(2.4.5)]{kilbas}, where it is a theorem and not
the definition: on $AC$ the two coincide. Identity \eqref{eq:RLvsC} follows from
\eqref{eq:caputodef} and $\bigl(\DL c\bigr)(t)=c\,t^{-\alpha}/\Gamma(1-\alpha)$ for a constant
$c$, and \eqref{eq:ftfc} is the fundamental theorem of fractional calculus,
\cite[Lemma~2.22]{kilbas}; see also \cite[Ch.~3]{diethelm}, \cite[Sec.~2.4]{podlubny}.
\end{proof}

Identity \eqref{eq:RLvsC} separates the two possible readings of $\partial^{\alpha}_{t}$ on the
scale of Definition~\ref{def:space}. The function $t\mapsto t^{-\alpha}$ belongs to $L^{2}(0,T)$
if and only if $\alpha<\frac12$; see Remark~\ref{rem:pairunique}.

We let $\RLD$ act only on test functions and never on the solution: by \eqref{eq:RLderright} its
value at $t$ is determined by $\phi$ on $[t,T]$. Differentiating under the integral sign, which
is licit because $\phi\in C^{\infty}_{0}\bigl((0,T);\HH\bigr)$ has compact support, and using
$\phi(T)=0$, we obtain
\begin{equation}\label{eq:RLDrep}
\bigl(\RLD\phi\bigr)(t)=-\frac{1}{\Gamma(1-\alpha)}\int_{0}^{T-t}
\sigma^{-\alpha}\,\phi'(t+\sigma)\,d\sigma ,
\end{equation}
Hence $\IR{1-\alpha}\phi\in C^{1}\bigl([0,T];\HH\bigr)$, and the function $\RLD\phi$ is
continuous on $[0,T]$ and vanishes to the right of the support of $\phi$. Moreover
\begin{equation}\label{eq:rightRLbound}
\norm{\RLD\phi}_{L^{\infty}((0,T);\HH)}
\le\frac{T^{1-\alpha}}{\Gamma(2-\alpha)}\norm{\partial_{t}\phi}_{L^{\infty}((0,T);\HH)} .
\end{equation}
In particular $\RLD\phi\in L^{q}\bigl((0,T);\HH\bigr)$ for every $q\in[1,\infty]$, a fact we use
in the passage to the limit.

\begin{lemma}[Fractional integration by parts for
\texorpdfstring{$\partial^{\alpha}_{t}$}{the Caputo derivative}]\label{lem:ibp}
Let $v\in AC\bigl([0,T];\LL\bigr)$ and $\phi\in C^{\infty}_{0}\bigl((0,T);\HH\bigr)$. Then
\begin{equation}\label{eq:intbyparts}
\int_{0}^{T}\bigl\langle\partial^{\alpha}_{t}v,\phi\bigr\rangle\,dt
=\int_{0}^{T}\bigl\langle v,\RLD\phi\bigr\rangle\,dt
-\bigl\langle v(0),\bigl(\IR{1-\alpha}\phi\bigr)(0)\bigr\rangle .
\end{equation}
\end{lemma}

\begin{proof}
This is the integration by parts of \cite[Sec.~3]{kubica}; see also \cite[Ch.~2]{kubicabook}. It
combines \eqref{eq:caputo}, Lemma~\ref{lem:RLprops}(iv) with $\gamma=1-\alpha$, the classical
integration by parts, which \eqref{eq:RLDrep} makes available, and \eqref{eq:RLderright}. The
endpoint term at $T$ disappears because $\bigl(\IR{1-\alpha}\phi\bigr)(T)=0$.
\end{proof}

\begin{remark}[The datum survives the compact support]\label{rem:datumterm}
The term that \eqref{eq:intbyparts} leaves behind is the \emph{right-sided} integral evaluated at
the \emph{left} endpoint,
\[
\bigl(\IR{1-\alpha}\phi\bigr)(0)=\frac{1}{\Gamma(1-\alpha)}\int_{0}^{T}s^{-\alpha}\phi(s)\,ds ,
\]
which does not vanish although $\operatorname{supp}\phi\Subset(0,T)$; the left-sided integral
evaluated at $0$ would vanish identically, and with it the initial datum would disappear from the
weak formulation \eqref{eq:weak} and every solution would be independent of $g$.
\end{remark}

\subsection{The solution space and the weak formulation}

\begin{definition}[Solution space]\label{def:space}
Let $u\in\LHone$. We say that a pair $(\chi,u_{0})\in\LHmone\times\LL$ \emph{represents the
Caputo derivative of $u$ with initial value $u_{0}$} if
\begin{equation}\label{eq:spacedef}
\int_{0}^{T}\langle\chi,\phi\rangle\,dt
=\int_{0}^{T}\bigl\langle u,\RLD\phi\bigr\rangle\,dt
-\bigl\langle u_{0},\bigl(\IR{1-\alpha}\phi\bigr)(0)\bigr\rangle
\qquad\forall\,\phi\in C^{\infty}_{0}\bigl((0,T);\HH\bigr),
\end{equation}
and we then write $\chi=\partial^{\alpha}_{t}u$ and $u(0)=u_{0}$; by Lemma~\ref{lem:ibp} this is
consistent with Definition~\ref{def:caputo}, the classical pair qualifying whenever
$u\in AC\bigl([0,T];\LL\bigr)$. We set
\begin{equation}\label{eq:space}
\Wscale:=\bigl\{u\in\LHone:\ \text{some pair }(\chi,u_{0})\text{ satisfies }
\eqref{eq:spacedef}\bigr\},
\end{equation}
which is the scale of Kubica and Yamamoto \cite{kubica}, to which we refer for its completeness.
Alongside it we use, for a prescribed datum $g\in\LL$, the subset of \eqref{eq:space} on which
$g$ is the represented initial value: we write $\WscaleG$ for the set of those $u\in\LHone$ for
which \eqref{eq:spacedef} holds with the pair $(\chi,g)$ for some $\chi\in\LHmone$, and we norm it
by $\norm{u}_{\LHone}+\norm{\chi}_{\LHmone}$. On $\WscaleG$ that second summand is unambiguous for
every $\alpha\in(0,1)$, because by Remark~\ref{rem:pairunique} two admissible pairs sharing their
initial value share their first entry as well; on \eqref{eq:space} it is unambiguous only for
$\alpha\ge\frac12$. Every statement below that places a solution in a normed space places it in
$\WscaleG$, with the datum of \eqref{eq:P}, which is the datum the weak formulation
\eqref{eq:weak} prescribes, so that nothing is asserted about a norm the order could leave
undetermined.
\end{definition}

\begin{remark}[The pair, and a third occurrence of
\texorpdfstring{$\alpha=\frac12$}{alpha=1/2}]\label{rem:pairunique}
Suppose $(\chi,u_{0})$ and $(\tilde\chi,\tilde u_{0})$ both satisfy \eqref{eq:spacedef} for the
same $u$, and write $h:=u_{0}-\tilde u_{0}$. Subtracting the two identities,
\[
\int_{0}^{T}\langle\chi-\tilde\chi,\phi\rangle\,dt
=-\bigl\langle h,\bigl(\IR{1-\alpha}\phi\bigr)(0)\bigr\rangle
=-\frac{1}{\Gamma(1-\alpha)}\int_{0}^{T}t^{-\alpha}\langle h,\phi(t)\rangle\,dt ,
\]
that is $\chi-\tilde\chi=-t^{-\alpha}h/\Gamma(1-\alpha)$ in $\mathcal{D}'\bigl((0,T);\HM\bigr)$,
which is exactly the discrepancy \eqref{eq:RLvsC} between the Riemann--Liouville and the Caputo
readings. That function lies in $\LHmone$ if and only if $h=0$ or $\alpha<\frac12$. For $\alpha\in[\frac12,1)$ the pair is therefore determined by $u$ alone, and the norm on
\eqref{eq:space} is unambiguous. For $\alpha\in(0,\frac12)$ it is not: two pairs with different
initial values may represent one and the same $u$. Prescribing the initial value removes the
indeterminacy at every order, since $h=0$ forces $\chi=\tilde\chi$ whatever $\alpha$ may be, and
that is the reason the space $\WscaleG$ carries a norm for all $\alpha\in(0,1)$ while
\eqref{eq:space} does not. Prescribing it is also what the weak formulation \eqref{eq:weak} does,
the datum appearing on its right-hand side. The uniqueness proof of
Theorem~\ref{thm:uniqueness} is in any case independent of all this: it never uses the pair, only
the weak formulation itself.

We follow \cite{kubica} in building the scale on $\partial^{\alpha}_{t}$ rather than on $\DL$,
though the choice is not forced. The same computation shows what it buys. By \eqref{eq:RLvsC} the
Riemann--Liouville reading differs from the Caputo one by $t^{-\alpha}u(0)/\Gamma(1-\alpha)$,
which for $u(0)=g\neq0$ leaves $\LHmone$ exactly when $\alpha\ge\frac12$, so the solutions
constructed below would fall outside the space for every order above one half. This occurrence of
$\alpha=\frac12$ concerns the representation of the derivative, and is distinct from the two the
paper is about; see Section~\ref{sec:comparison}.
\end{remark}

\begin{definition}[Weak solution]\label{def:weak}
A function $u\in\WscaleG$ is a weak solution of \eqref{eq:P} if
\begin{equation}\label{eq:weak}
\int_{0}^{T}\!\!\int_{\Omega}u\,\RLD\phi\,dx\,dt
+\int_{0}^{T}\!\!\int_{\Omega}A\nabla u\cdot\nabla\phi\,dx\,dt
=\int_{0}^{T}\!\!\int_{\Omega}f\phi\,dx\,dt
+\int_{\Omega}g\,\bigl(\IR{1-\alpha}\phi\bigr)(x,0)\,dx
\end{equation}
for every $\phi\in C^{\infty}_{0}\bigl((0,T);\HH\bigr)$.
\end{definition}

We do not impose the initial datum as a trace, since a function of $\Wscale$ need not have one.
It enters instead through the last term of \eqref{eq:weak}, which is what the integration by
parts \eqref{eq:intbyparts} leaves behind, and which by Remark~\ref{rem:pairunique} prescribes
the initial value of the pair representing $\partial^{\alpha}_{t}u$.

\subsection{The hypotheses}\label{sec:hyp}

Hypothesis \Hone{} is what makes the embedding $\HH\hookrightarrow\LL$ compact, by
Rellich--Kondrachov \cite[Thm.~9.16]{brezis}, and hence what makes the spectral basis of
Lemma~\ref{lem:spectral} and the compactness argument of Section~\ref{sec:common} available.
Hypothesis \Htwo{} gives the coercivity and boundedness \eqref{eq:coercive}; the scalar
operator $-\nabla\cdot\bigl(C^{2}(x)\nabla u\bigr)$ with $C\in L^{\infty}(\Omega)$ bounded
below by $c_{0}>0$ is a particular case $A=C^{2}I$, with $a_{0}=c_{0}^{2}$ and
$a_{1}=\norm{C}^{2}_{L^{\infty}}$. Hypothesis \Hfour{}
is the minimal requirement for the datum term of \eqref{eq:weak} to make sense. The two alternatives
\Hthree{} and \Hthreep{} are compared in the next section. Note that \Hthreep{} is existential:
it fixes no exponent, and demands only that some exponent exceed the threshold $2/\alpha$.

\subsection{The relation between \texorpdfstring{\Hthree}{(H3)} and
\texorpdfstring{\Hthreep}{(H3')}}\label{sec:relation}

We work throughout in the range $0<\alpha<1$, and in that range the threshold $2/\alpha$ is
strictly larger than $2$. This single arithmetic fact orders the two hypotheses.

\begin{proposition}[Hierarchy of the source hypotheses]\label{prop:hierarchy}
Let $\alpha\in(0,1)$ and $T<\infty$.
\begin{enumerate}[label=\textup{(\roman*)},leftmargin=2.4em]
\item $\Hthreep\Longrightarrow\Hthree$: if $f\in\LLr$ with $r>2/\alpha$, then $f\in\LLtwo$ and
\begin{equation}\label{eq:inclusion}
\norm{f}_{\LLtwo}\le T^{\frac12-\frac1r}\norm{f}_{\LLr};
\end{equation}
\item $\Hthree\not\Longrightarrow\Hthreep$: if $w\in\LL$ has $\norm{w}_{\LL}=1$, the source
$f(x,t):=t^{-\alpha/2}w(x)$ satisfies \Hthree{} and lies in $\LLr$ for no $r>2/\alpha$.
\end{enumerate}
\end{proposition}

\begin{proof}
(i) Since $\alpha<1$, every admissible exponent satisfies $r>2/\alpha>2$. We put $p:=r/2>1$ and
let $p'$ be its conjugate. H\"older's inequality on the bounded interval $(0,T)$, applied to
$\int_{0}^{T}\norm{f(t)}^{2}_{\LL}\cdot1\,dt$, gives
\[
\int_{0}^{T}\norm{f(t)}^{2}_{\LL}dt
\le\left(\int_{0}^{T}\norm{f(t)}^{2p}_{\LL}dt\right)^{1/p}T^{1/p'}
=\norm{f}^{2}_{\LLr}\,T^{1-2/r},
\]
which is \eqref{eq:inclusion} after taking square roots.

(ii) Here $\norm{f(t)}_{\LL}=t^{-\alpha/2}$, so that
$\norm{f}^{2}_{\LLtwo}=\int_{0}^{T}t^{-\alpha}dt<\infty$ because $\alpha<1$, and $f$ satisfies
\Hthree. For $q\ge1$ we have $\norm{f}^{q}_{L^{q}((0,T);\LL)}=\int_{0}^{T}t^{-q\alpha/2}dt$,
finite exactly when $q\alpha/2<1$, that is when $q<2/\alpha$. No exponent above the threshold is
therefore admissible, and \Hthreep{} fails for this $f$.
\end{proof}

Hypothesis \Hthree{} is therefore strictly weaker than \Hthreep, and two consequences will be
used repeatedly. Everything proved under \Hthree{} is available under
\Hthreep{} as well, so in the common development of Section~\ref{sec:common} we may assume
$f\in\LLtwo$ once and for all. The implication does not reverse, and the gap between the two
hypotheses has mathematical content. Section~\ref{sec:source} shows that for sources admitted by
\Hthree{} alone the quantity an energy argument must control is unbounded, and
Section~\ref{sec:sharp} that for $\alpha\le\frac12$ the conclusion $u\in\LinfL$ is false for some
of them.

\begin{remark}[Correction of a claim in the literature]\label{rem:correction}
It is sometimes asserted, for this pair of hypotheses, that neither implies the other, with a
function $f\in L^{r}$ having $1/\alpha<r<2$ offered as a witness against
$\Hthreep\Rightarrow\Hthree$. The assertion is false in the range considered here and the witness
does not qualify. Hypothesis \Hthreep{} requires $r>2/\alpha$, and since $\alpha<1$ this forces
$r>2$; an exponent below $2$ never satisfies \Hthreep{} in the first place, so no such function
is a counterexample to anything. For $\alpha\le\frac12$ the interval $(1/\alpha,2)$ is empty, so in half of the range the
proposed witness cannot be written down at all. The
correct statement is the one we prove in Proposition~\ref{prop:hierarchy}:
$\Hthreep\Rightarrow\Hthree$ always, and the converse fails. The hypotheses are ordered, and the
order is what dictates the arrangement of Sections~\ref{sec:caseA} and~\ref{sec:caseB}: the
weaker hypothesis first.
\end{remark}

That the two theorems we prove here are not ordered, although their hypotheses are, is a separate
matter, and we discuss it in Section~\ref{sec:comparison} once both have been proved.

\section{The common development}\label{sec:common}

Here we assemble, once, everything that both cases use: the analytic tools, the Galerkin scheme,
the energy inequality from which both routes start, the exact point at which they separate, the compactness machinery, uniqueness, and a bound on the approximations that the energy route
does not give. By Proposition~\ref{prop:hierarchy}(i) the weaker
hypothesis \Hthree{} holds in both cases, so we assume it throughout together with \Hone, \Htwo{}
and \Hfour, and we use nothing stronger anywhere in this section. Sections~\ref{sec:caseA}
and~\ref{sec:caseB} then have only to supply uniform bounds, each from its own hypothesis.

\subsection{Analytic tools}\label{sec:tools}

The only property of the kernel $(t-s)^{\gamma-1}$ we need is its integrability near
the singularity: a power $\sigma^{m}$ is integrable at the origin exactly when $m>-1$, and then
$\int_{0}^{a}\sigma^{m}d\sigma=a^{m+1}/(m+1)$. Applied to exponents fixed by the order of the
equation this is arithmetic. It becomes a genuine condition only twice, when the exponent carries
a free factor, and those two occasions are where the threshold $r>2/\alpha$ and the window
$p<2\alpha$ come from.

\begin{lemma}[Mittag-Leffler functions]\label{lem:ML}
For $\gamma,\delta>0$ let $E_{\gamma,\delta}(z):=\sum_{k\ge0}z^{k}/\Gamma(\gamma k+\delta)$, and
write $E_{\gamma}:=E_{\gamma,1}$. Then
\begin{enumerate}[label=\textup{(\roman*)},leftmargin=2.4em]
\item $E_{\gamma,\delta}$ is entire, hence bounded on compact subsets of $\R$;
\item if $\gamma\in(0,1]$ and $\delta\ge\gamma$, then $x\mapsto E_{\gamma,\delta}(-x)$ is
completely monotone on $[0,\infty)$, and in particular
\begin{equation}\label{eq:MLbound}
0<E_{\gamma,\delta}(-x)\le\frac{1}{\Gamma(\delta)}\qquad(x\ge0);
\end{equation}
\item for every $\lambda>0$,
\begin{equation}\label{eq:MLder}
\frac{d}{dt}E_{\gamma,1}\bigl(-\lambda t^{\gamma}\bigr)
=-\lambda\,t^{\gamma-1}E_{\gamma,\gamma}\bigl(-\lambda t^{\gamma}\bigr),\qquad t>0 .
\end{equation}
\end{enumerate}
\end{lemma}

\begin{proof}
For the definition and for (i) see \cite[Sec.~1.8]{kilbas}, \cite[Sec.~1.2]{podlubny} and
\cite[Ch.~4]{gorenflo}. Part (ii) is the Pollard--Schneider theorem, \cite[Ch.~4]{gorenflo}.
Part (iii) follows by differentiating the series term by term, which is licit because it is
entire:
\[
\frac{d}{dt}E_{\gamma,1}\bigl(-\lambda t^{\gamma}\bigr)
=\sum_{k\ge1}\frac{(-\lambda)^{k}\gamma k\,t^{\gamma k-1}}{\Gamma(\gamma k+1)}
=-\lambda t^{\gamma-1}\sum_{m\ge0}\frac{\bigl(-\lambda t^{\gamma}\bigr)^{m}}
{\Gamma(\gamma m+\gamma)}
=-\lambda t^{\gamma-1}E_{\gamma,\gamma}\bigl(-\lambda t^{\gamma}\bigr).
\]
\end{proof}

\begin{lemma}[Alikhanov inequality]\label{lem:alikhanov}
Let $\gamma\in(0,1)$ and $v\in AC\bigl([0,T];\LL\bigr)$. Then
\begin{equation}\label{eq:alikhanov}
\bigl(\partial^{\gamma}_{t}v(t),v(t)\bigr)_{\LL}
\ge\tfrac12\,\partial^{\gamma}_{t}\norm{v(t)}^{2}_{\LL},\qquad\text{a.e. }t\in(0,T).
\end{equation}
\end{lemma}

\begin{proof}
See \cite[Lemma~1]{alikhanov}, and \cite[Ch.~2]{kubicabook} for the Hilbert-space valued version.
\end{proof}

\begin{lemma}[Fractional Gr\"onwall inequality]\label{lem:gronwall}
Let $\gamma>0$ and $\varepsilon>0$, let $a:[0,T]\to[0,\infty)$ be non-decreasing and locally
integrable and let $y:[0,T]\to[0,\infty)$ be locally integrable with
\begin{equation}\label{eq:gronwall-hyp}
y(t)\le a(t)+\varepsilon\,\bigl(\IL{\gamma}y\bigr)(t),\qquad t\in[0,T].
\end{equation}
Then $y(t)\le a(t)E_{\gamma}\bigl(\varepsilon t^{\gamma}\bigr)$ on $[0,T]$.
\end{lemma}

\begin{proof}
See \cite[Thm.~1 and Cor.~2]{ye}. In \cite{ye} the hypothesis is written without the factor
$1/\Gamma(\gamma)$, so that \eqref{eq:gronwall-hyp} corresponds to $b=\varepsilon/\Gamma(\gamma)$
there and the argument of $E_{\gamma}$ in the conclusion is $\varepsilon t^{\gamma}$, with no
factor $\Gamma(\gamma)$. The constants below are computed with this normalisation.
\end{proof}

Under \Hthree{} there is no pointwise control in $t$, and the energy inequality will have to be
integrated in time before Gr\"onwall is applied. Integration could destroy the Gr\"onwall
structure. It does not, because integrating $\IL{\gamma}h$ produces a fractional integral of the
primitive of $h$.

\begin{lemma}[Integration in time]\label{lem:timeint}
Let $\gamma>0$, $\tau\in(0,T]$ and $h\in L^{1}(0,\tau)$, and let
$H(\sigma):=\int_{0}^{\sigma}h$. Then
\begin{enumerate}[label=\textup{(\roman*)},leftmargin=2.4em]
\item the identity
\begin{equation}\label{eq:timeint}
\int_{0}^{\tau}\bigl(\IL{\gamma}h\bigr)(t)\,dt
=\frac{1}{\Gamma(\gamma+1)}\int_{0}^{\tau}h(s)(\tau-s)^{\gamma}ds
=\bigl(\IL{\gamma}H\bigr)(\tau)
\end{equation}
holds;
\item if in addition $h\ge0$, then
$\int_{0}^{\tau}\bigl(\IL{\gamma}h\bigr)(t)\,dt\le\tau^{\gamma}
\Gamma(\gamma+1)^{-1}\norm{h}_{L^{1}(0,\tau)}$.
\end{enumerate}
\end{lemma}

\begin{proof}
The three expressions are one and the same object, $\bigl(\IL{\gamma+1}h\bigr)(\tau)$.
Integrating over $(0,\tau)$ amounts to applying $\IL{1}$, so from the semigroup property of
Lemma~\ref{lem:RLprops}(ii) we obtain
$\int_{0}^{\tau}\bigl(\IL{\gamma}h\bigr)(t)\,dt=\bigl(\IL{\gamma+1}h\bigr)(\tau)$, which written
out is the middle expression; applying the semigroup in the other order, with
$H=\IL{1}h$, we obtain
$\bigl(\IL{\gamma}H\bigr)(\tau)=\bigl(\IL{\gamma+1}h\bigr)(\tau)$. All three are continuous in
$\tau$, so the identity holds at every point and not merely almost everywhere. The final bound
follows from $(\tau-s)^{\gamma}\le\tau^{\gamma}$.
\end{proof}

\subsection{The Galerkin scheme}\label{sec:galerkin}

\begin{lemma}[Spectral basis]\label{lem:spectral}
Under \Hone--\Htwo{} the operator $A:\HH\to\HM$ defined by $\langle Au,v\rangle:=a(u,v)$ is
linear, continuous, self-adjoint and coercive, with compact resolvent. There are therefore
eigenvalues $0<\lambda_{1}\le\lambda_{2}\le\cdots\to\infty$ and eigenfunctions
$\{w_{j}\}_{j\ge1}\subset\HH$ with
\begin{equation}\label{eq:eigen}
a(w_{j},v)=\lambda_{j}(w_{j},v)_{\LL}\qquad\forall v\in\HH,\ \forall j\ge1,
\end{equation}
the family $\{w_{j}\}$ being an orthonormal basis of $\LL$ and an orthogonal basis of $\HH$; in
particular $a(w_{j},w_{k})=\lambda_{j}\delta_{jk}$.
\end{lemma}

\begin{proof}
This is the Lax--Milgram lemma together with the Rellich--Kondrachov theorem, available because
$\partial\Omega$ is Lipschitz, \cite[Thm.~9.16]{brezis}, and the spectral theorem for compact
self-adjoint operators, \cite[Thm.~6.11]{brezis}.
\end{proof}

For $n\in\mathbb{N}$ we set $V_{n}:=\operatorname{span}\{w_{1},\dots,w_{n}\}$ and look for
\begin{equation}\label{eq:un}
u_{n}(t):=\sum_{j=1}^{n}c_{j,n}(t)\,w_{j}
\end{equation}
solving the equation projected onto $V_{n}$: for $j=1,\dots,n$ and a.e. $t\in(0,T)$,
\begin{equation}\label{eq:galerkin}
\bigl(\partial^{\alpha}_{t}u_{n}(t),w_{j}\bigr)_{\LL}+a\bigl(u_{n}(t),w_{j}\bigr)
=\bigl(f(t),w_{j}\bigr)_{\LL} .
\end{equation}
Orthonormality in $\LL$ and $a(w_{k},w_{j})=\lambda_{j}\delta_{jk}$ let us identify the sums
termwise, so that \eqref{eq:galerkin} is equivalent to $n$ decoupled scalar equations,
\begin{equation}\label{eq:modes}
\partial^{\alpha}_{t}c_{j,n}(t)+\lambda_{j}c_{j,n}(t)=f_{j,n}(t),
\qquad f_{j,n}:=\bigl(f(\cdot),w_{j}\bigr)_{\LL},
\end{equation}
where we take the initial values from the orthogonal projection of the datum,
\begin{equation}\label{eq:gn}
c_{j,n}(0)=(g,w_{j})_{\LL},\qquad g_{n}:=\sum_{j\le n}(g,w_{j})_{\LL}w_{j} .
\end{equation}
These projections are onto increasing subspaces whose union is dense, so that
\begin{equation}\label{eq:gnconv}
g_{n}\to g\ \text{ in }\LL,\qquad\norm{g_{n}}_{\LL}\le\norm{g}_{\LL} .
\end{equation}
Observe that the system is genuinely decoupled: neither \eqref{eq:modes} nor \eqref{eq:gn}
involves $n$, so $c_{j,n}$ is the same function for every $n\ge j$, and $u_{n}$ is the $n$-th
partial sum of one fixed expansion. We use this in Section~\ref{sec:spectral}.

\begin{lemma}[The scalar fractional equation]\label{lem:scalar}
Let $\lambda>0$, $F\in L^{1}(0,T)$ and $\eta_{0}\in\R$, and let us consider the Volterra
equation
\begin{equation}\label{eq:volterra}
c(t)+\lambda\bigl(\IL{\alpha}c\bigr)(t)=\eta_{0}+\bigl(\IL{\alpha}F\bigr)(t),
\qquad\text{a.e. }t\in(0,T).
\end{equation}
\begin{enumerate}[label=\textup{(\roman*)},leftmargin=2.4em]
\item The function
\begin{equation}\label{eq:cexplicit}
c(t)=\eta_{0}E_{\alpha,1}\bigl(-\lambda t^{\alpha}\bigr)+(K*F)(t),
\qquad K(\sigma):=\sigma^{\alpha-1}E_{\alpha,\alpha}\bigl(-\lambda\sigma^{\alpha}\bigr),
\end{equation}
solves \eqref{eq:volterra}, and $c\in C([0,T])$ whenever $F\in L^{r}(0,T)$ with $r>1/\alpha$;
\item equation \eqref{eq:volterra} has at most one solution in $L^{1}(0,T)$;
\item if moreover $c\in AC([0,T])$, then \eqref{eq:volterra} is equivalent to
$\partial^{\alpha}_{t}c+\lambda c=F$ a.e. together with $c(0)=\eta_{0}$.
\end{enumerate}
\end{lemma}

\begin{proof}
Part (i) is Laplace operational calculus, \cite[Secs.~1.2, 4.1]{podlubny},
\cite[Sec.~4.1]{kilbas}, \cite[Sec.~2]{sakamoto}, the regularity being the mapping property of
$\IL{\alpha}$ recorded in Proposition~\ref{prop:threshold}(i). Part (ii) is the resolvent theory
of linear Volterra equations of the second kind with locally integrable kernel,
\cite[Ch.~2]{gripenberg}, \cite[Ch.~I]{pruss}. Part (iii) is \cite[Lemma~6.2]{diethelm} for
continuous $F$, and \cite[Ch.~2]{gripenberg}, \cite[Ch.~2]{kubicabook} in the generality needed
here.
\end{proof}

Applying parts (i) and (ii) mode by mode with $\lambda=\lambda_{j}$ and $F=f_{j,n}$, we determine
$c_{j,n}$, and with it $u_{n}$, uniquely. Since $\abs{f_{j,n}(t)}\le\norm{f(t)}_{\LL}$ by
Cauchy--Schwarz, each mode inherits the integrability of $f$; under \Hthree{} with
$\alpha>\frac12$ the exponent $r=2$ already satisfies $r>1/\alpha$, and under \Hthreep{} we have
$r>2/\alpha>1/\alpha$, so that in both of those cases
\begin{equation}\label{eq:unC}
u_{n}\in C\bigl([0,T];V_{n}\bigr)
\qquad\text{(under \Hthree{} with }\alpha>\tfrac12,\text{ or under \Hthreep).}
\end{equation}

The integral form \eqref{eq:volterra} mentions no derivative, so we must say what
$\partial^{\alpha}_{t}u_{n}$ denotes in \eqref{eq:galerkin}. We define it by solving for it in
\eqref{eq:modes},
\begin{equation}\label{eq:defdalpha}
\partial^{\alpha}_{t}u_{n}:=f_{n}-Au_{n}\in L^{2}\bigl((0,T);V_{n}\bigr),
\qquad f_{n}:=\sum_{j\le n}f_{j,n}w_{j},
\end{equation}
so that \eqref{eq:galerkin} holds by construction, for every admissible $f$ and with no extra
regularity. This is not an abuse of notation: by Lemma~\ref{lem:scalar}(iii), whenever $u_{n}$ is
absolutely continuous the object \eqref{eq:defdalpha} is the Caputo derivative of
Definition~\ref{def:caputo}.

\begin{lemma}[Regular modes and extension by density]\label{lem:modes}
Let $\mathcal{D}:=C^{\infty}_{0}\bigl((0,T);\LL\bigr)$ and let $n\in\mathbb{N}$ be fixed. Then
\begin{enumerate}[label=\textup{(\roman*)},leftmargin=2.4em]
\item $\mathcal{D}$ is dense in $L^{r}\bigl((0,T);\LL\bigr)$ for every $r\in[1,\infty)$;
\item if $f\in\mathcal{D}$, then $u_{n}\in AC\bigl([0,T];V_{n}\bigr)$, the object
\eqref{eq:defdalpha} is the classical Caputo derivative, and \eqref{eq:galerkin} holds
pointwise;
\item the map $f\mapsto u_{n}$ is continuous from $\LLtwo$ to $L^{2}\bigl((0,T);V_{n}\bigr)$ for
every $\alpha\in(0,1)$, and from $L^{r}\bigl((0,T);\LL\bigr)$ to $C\bigl([0,T];V_{n}\bigr)$
whenever $r>1/\alpha$.
\end{enumerate}
\end{lemma}

\begin{proof}
(i) is the density of $C^{\infty}_{0}$ in $L^{r}$ for $r<\infty$.

(ii) Since $V_{n}$ is finite-dimensional it suffices to treat one coefficient; we write $c$, $F$,
$\lambda$, $K$ as in \eqref{eq:cexplicit} and note first that $K\in L^{1}(0,T)$, because
$E_{\alpha,\alpha}$ is bounded on $[-\lambda T^{\alpha},0]$ and $\sigma^{\alpha-1}$ is integrable.
For the source part, $f\in\mathcal{D}$ makes $F$ smooth with $F(0)=0$, so that $(K*F)'=K*F'$ is
continuous \cite[Ch.~2]{gripenberg}; iterating, we obtain $K*F\in C^{\infty}([0,T])$. For the
data part, \eqref{eq:MLder} gives
$\frac{d}{dt}E_{\alpha,1}(-\lambda t^{\alpha})=-\lambda K(t)\in L^{1}(0,T)$, so that
$t\mapsto E_{\alpha,1}(-\lambda t^{\alpha})$ is absolutely continuous. Adding the two, we obtain
$c\in AC([0,T])$, and Lemma~\ref{lem:scalar}(iii) finishes the proof.

(iii) Only the dependence on $F$ is at issue, and by \eqref{eq:cexplicit} it is confined to
$K*F$. From Young's inequality for convolutions \cite[Thm.~4.15]{brezis} we obtain
$\norm{K*F}_{L^{2}}\le\norm{K}_{L^{1}}\norm{F}_{L^{2}}$, and $K\in L^{1}(0,T)$ for every
$\alpha>0$: integrability of the kernel is all we need. Control of the pointwise value is more
demanding. H\"older with conjugate exponents $r'$ and $r$ gives
$\abs{(K*F)(t)}\le\norm{K}_{L^{r'}(0,T)}\norm{F}_{L^{r}(0,T)}$, and $\norm{K}_{L^{r'}}$ is finite
exactly when $r'(\alpha-1)>-1$, that is when $r>1/\alpha$. Here the threshold is born: for the
averaged statement $K\in L^{1}$ suffices, while for the pointwise one we need $K\in L^{r'}$.
\end{proof}

\begin{remark}[How the density argument is used, and where]\label{rem:density}
Two of the tools below require $u_{n}\in AC\bigl([0,T];\LL\bigr)$, which \eqref{eq:unC} does not
give: the Alikhanov inequality and the integration by parts \eqref{eq:intbyparts}. The usual remedy applies in two
stages. We first prove everything for $f\in\mathcal{D}$, where Lemma~\ref{lem:modes}(ii) supplies
the regularity, and then extend by density. The extension is legitimate because the first stage
produces either an inequality $\norm{u_{n}}_{X}\le\Phi\bigl(\norm{f},\norm{g}\bigr)$, with $\Phi$
continuous and $\norm{\cdot}_{X}$ weakly lower semicontinuous, or an identity linear in
$(u_{n},\partial^{\alpha}_{t}u_{n},f)$. In the
case \Hthree{} we use only the $L^{2}$ branch of Lemma~\ref{lem:modes}(iii), which restricts
nothing; in the case \Hthreep{} a supremum in $t$ is involved and we use the $C([0,T])$ branch,
whose requirement $r>1/\alpha$ is met since $r>2/\alpha$.

From here on the energy estimates of Sections~\ref{sec:basic} and~\ref{sec:source}, and those of
Sections~\ref{sec:caseA} and~\ref{sec:caseB} that rest on them, are proved for $f\in\mathcal{D}$
and extended by density without further mention. Section~\ref{sec:spectral} is not subject to
this reduction: it works with the explicit representation \eqref{eq:modesofu}, which needs no
absolute continuity, and it states its own hypotheses on $f$. Neither is
Proposition~\ref{prop:sharp}(ii), whose source is constructed, not approximated.
\end{remark}

\subsection{The basic energy inequality}\label{sec:basic}

We multiply the $j$-th equation of \eqref{eq:modes} by $c_{j,n}(t)$ and sum over $j\le n$; the
same two identifications as before give
\begin{equation}\label{eq:basic}
\bigl(\partial^{\alpha}_{t}u_{n}(t),u_{n}(t)\bigr)_{\LL}+a\bigl(u_{n}(t),u_{n}(t)\bigr)
=\bigl(f(t),u_{n}(t)\bigr)_{\LL},\qquad\text{a.e. }t\in(0,T).
\end{equation}
We bound the three terms in turn. The first from below by Lemma~\ref{lem:alikhanov} with
$\gamma=\alpha$, applicable because $u_{n}\in AC\bigl([0,T];\LL\bigr)$ by
Lemma~\ref{lem:modes}(ii), which requires $f\in\mathcal{D}$; we work under that restriction for
the remainder of this subsection and remove it in Lemma~\ref{lem:dens}. The second from below by
coercivity \eqref{eq:coercive}; the third from
above by Cauchy--Schwarz followed by Young's inequality with a parameter $\varepsilon>0$,
\[
\bigl(f(t),u_{n}(t)\bigr)_{\LL}\le\frac{1}{2\varepsilon}\norm{f(t)}^{2}_{\LL}
+\frac{\varepsilon}{2}\norm{u_{n}(t)}^{2}_{\LL} .
\]
We keep $\varepsilon$ free because the stability constant of Theorem~\ref{thm:stability} depends
on it. Substituting and multiplying by $2$, we obtain
\begin{equation}\label{eq:diffineq}
\partial^{\alpha}_{t}\norm{u_{n}(t)}^{2}_{\LL}+2a_{0}\norm{u_{n}(t)}^{2}_{\HH}
\le\frac{1}{\varepsilon}\norm{f(t)}^{2}_{\LL}+\varepsilon\norm{u_{n}(t)}^{2}_{\LL} .
\end{equation}
From here on $\partial^{\alpha}_{t}$ acts on the scalar function $\norm{u_{n}(t)}^{2}_{\LL}$
rather than on a function with values in $V_{n}$, so that \eqref{eq:diffineq} is a scalar
fractional differential inequality, the form Lemma~\ref{lem:gronwall} requires.

We now apply $\IL{\alpha}$ to both sides. The inequality is preserved because $\IL{\alpha}$ is
order preserving, Lemma~\ref{lem:RLprops}(iii), and we treat the first term by \eqref{eq:ftfc}
with $u=\norm{u_{n}(\cdot)}^{2}_{\LL}$, which is absolutely continuous: it equals
$\sum_{j\le n}c_{j,n}^{2}$, each $c_{j,n}$ is absolutely continuous and bounded, and products of
such functions are again absolutely continuous. Since
$\norm{u_{n}(0)}^{2}_{\LL}=\norm{g_{n}}^{2}_{\LL}$, we obtain
\begin{equation}\label{eq:integralform}
\norm{u_{n}(t)}^{2}_{\LL}-\norm{g_{n}}^{2}_{\LL}
+2a_{0}\bigl(\IL{\alpha}\norm{u_{n}}^{2}_{\HH}\bigr)(t)
\le\frac{1}{\varepsilon}\bigl(\IL{\alpha}\norm{f}^{2}_{\LL}\bigr)(t)
+\varepsilon\bigl(\IL{\alpha}\norm{u_{n}}^{2}_{\LL}\bigr)(t) .
\end{equation}
The elliptic term on the left is non-negative, so that we may do two different things with
\eqref{eq:integralform}, and we use both: discard that term and keep information on the
$\LL$ norm, or discard instead $\norm{u_{n}(t)}^{2}_{\LL}$ and keep information on the $\HH$
norm. That is,
\begin{align}
\norm{u_{n}(t)}^{2}_{\LL}&\le\norm{g_{n}}^{2}_{\LL}
+\frac{1}{\varepsilon}\bigl(\IL{\alpha}\norm{f}^{2}_{\LL}\bigr)(t)
+\varepsilon\bigl(\IL{\alpha}\norm{u_{n}}^{2}_{\LL}\bigr)(t),\label{eq:energyL2}\\[2pt]
2a_{0}\bigl(\IL{\alpha}\norm{u_{n}}^{2}_{\HH}\bigr)(t)&\le\norm{g_{n}}^{2}_{\LL}
+\frac{1}{\varepsilon}\bigl(\IL{\alpha}\norm{f}^{2}_{\LL}\bigr)(t)
+\varepsilon\bigl(\IL{\alpha}\norm{u_{n}}^{2}_{\LL}\bigr)(t) .\label{eq:energyH1}
\end{align}
Both inequalities contain the same term in the datum,
$\bigl(\IL{\alpha}\norm{f}^{2}_{\LL}\bigr)(t)$, and everything we can extract from them turns on
whether that term is bounded uniformly in $t$.

\begin{lemma}[Extension of the identities and of the energy inequalities]\label{lem:dens}
Let $n\in\mathbb{N}$ be fixed and let $f\in\LLtwo$. Then
\begin{enumerate}[label=\textup{(\roman*)},leftmargin=2.4em]
\item $\partial^{\alpha}_{t}u_{n}=P_{n}f-A_{n}u_{n}$ in $L^{2}\bigl((0,T);V_{n}\bigr)$, where
$P_{n}$ is the orthogonal projection of $\LL$ onto $V_{n}$ and $A_{n}$ the restriction to $V_{n}$
of the operator of Lemma~\ref{lem:spectral}; in consequence
$f\mapsto\partial^{\alpha}_{t}u_{n}$ is continuous from $\LLtwo$ into
$L^{2}\bigl((0,T);V_{n}\bigr)$;
\item for every $\psi\in C^{\infty}_{0}\bigl((0,T);\HH\bigr)$,
\begin{equation}\label{eq:ibp-n}
\int_{0}^{T}\langle\partial^{\alpha}_{t}u_{n},\psi\rangle\,dt
=\int_{0}^{T}\bigl\langle u_{n},\RLD\psi\bigr\rangle\,dt
-\bigl\langle u_{n}(0),\bigl(\IR{1-\alpha}\psi\bigr)(0)\bigr\rangle ;
\end{equation}
\item the energy inequalities \eqref{eq:energyL2} and \eqref{eq:energyH1} hold for a.e.
$t\in(0,T)$, with the same constants.
\end{enumerate}
\end{lemma}

\begin{proof}
(i) The $j$-th mode satisfies \eqref{eq:modes}, that is
$\partial^{\alpha}_{t}c_{j,n}=f_{j,n}-\lambda_{j}c_{j,n}$, which is the $j$-th coefficient of the
asserted identity; its right-hand side lies in $L^{2}(0,T)$ because $f_{j,n}$ does by
Cauchy--Schwarz and $c_{j,n}$ does by Lemma~\ref{lem:modes}(iii). On the finite-dimensional
$V_{n}$ the operators $P_{n}$ and $A_{n}$ are bounded, so the continuity follows from that of
$f\mapsto u_{n}$, which is the $L^{2}$ branch of Lemma~\ref{lem:modes}(iii).

(ii) Let $f_{m}\in\mathcal{D}$ with $f_{m}\to f$ in $\LLtwo$, which Lemma~\ref{lem:modes}(i)
provides, and write $u^{m}_{n}$ for the corresponding approximation. For each $m$ the identity
holds, by Lemma~\ref{lem:ibp} together with Lemma~\ref{lem:modes}(ii). Both sides are linear and
continuous in the triple $\bigl(\partial^{\alpha}_{t}u_{n},u_{n},u_{n}(0)\bigr)$ for the norms of
$L^{2}\bigl((0,T);V_{n}\bigr)$ twice over and of $\LL$ once: on the left by Cauchy--Schwarz, on
the right by \eqref{eq:rightRLbound} and because $\bigl(\IR{1-\alpha}\psi\bigr)(0)$ is one fixed
element of $\HH$. The first two entries converge, by (i) and by Lemma~\ref{lem:modes}(iii), and
the third does not move at all: $u^{m}_{n}(0)=g_{n}$ for every $m$ by \eqref{eq:gn}, the initial
value being independent of the source. The identity therefore passes to the limit.

(iii) We argue with the same sequence, for which each inequality holds by the derivation above,
and examine the four terms. From $u^{m}_{n}\to u_{n}$ in $L^{2}\bigl((0,T);\LL\bigr)$ we get
$\norm{u^{m}_{n}(\cdot)}^{2}_{\LL}\to\norm{u_{n}(\cdot)}^{2}_{\LL}$ in $L^{1}(0,T)$, and likewise
$\norm{f_{m}(\cdot)}^{2}_{\LL}\to\norm{f(\cdot)}^{2}_{\LL}$; the term in
$\norm{u^{m}_{n}}^{2}_{\HH}$ is handled the same way, the convergence of
Lemma~\ref{lem:modes}(iii) taking place in $V_{n}$, where the two norms are equivalent. Since
$\IL{\alpha}$ is bounded on $L^{1}(0,T)$ by Lemma~\ref{lem:RLprops}(i), the three fractional
integrals converge in $L^{1}(0,T)$ as well, while the datum term does not depend on $m$. Along a
subsequence all four terms therefore converge for a.e. $t$ at once, and the inequality survives
the limit.
\end{proof}

\begin{remark}[What the extension covers]\label{rem:denswhere}
Lemma~\ref{lem:dens} is the explicit form of the programme announced in
Remark~\ref{rem:density}, and it carries the whole weight of the appeals to absolute continuity
made later: the two energy inequalities, obtained above through Lemma~\ref{lem:alikhanov}, hold
for every $f\in\LLtwo$ by part (iii), and with them the a priori bounds of
Sections~\ref{sec:boundsA} and~\ref{sec:boundsB}, which follow from those inequalities and from
nothing else; the identification of the limit in Lemma~\ref{lem:identification} uses
\eqref{eq:ibp-n}, available now at the same generality by part (ii). Under \Hthreep{} the proof
applies verbatim with $\LLr$ in place of $\LLtwo$, the relevant branch of
Lemma~\ref{lem:modes}(iii) being then the one into $C\bigl([0,T];V_{n}\bigr)$, whose requirement
$r>1/\alpha$ holds because $r>2/\alpha$.
\end{remark}

\subsection{The term
\texorpdfstring{$\bigl(\IL{\alpha}\norm{f}^{2}_{\LL}\bigr)(t)$}{I(alpha)|f|2}}\label{sec:source}

Under \Hthree{} we know only that the integrand $h:=\norm{f(\cdot)}^{2}_{\LL}$ lies in
$L^{1}(0,T)$, and that is not enough. Nor is this for want of a sharper argument: the bound does
not exist, and it does not exist for any exponent below the threshold.

\begin{proposition}[No uniform bound below the threshold]\label{prop:nobound}
Let $\alpha\in(0,1)$ and let $1\le p<1/\alpha$. Then
\begin{enumerate}[label=\textup{(\roman*)},leftmargin=2.4em]
\item there is no constant $C(T,\alpha,p)$ with
\begin{equation}\label{eq:nobound}
\sup_{t\in(0,T]}\bigl(\IL{\alpha}h\bigr)(t)\le C(T,\alpha,p)\norm{h}_{L^{p}(0,T)}
\qquad\text{for every }0\le h\in L^{p}(0,T);
\end{equation}
\item consequently, for every $r\in[2,2/\alpha)$ there is no constant $C$ with
$\sup_{t\le T}\bigl(\IL{\alpha}\norm{f}^{2}_{\LL}\bigr)(t)\le C\norm{f}^{2}_{\LLr}$ for every
$f\in\LLr$. The case $r=2$ is \Hthree.
\end{enumerate}
\end{proposition}

\begin{proof}
(i) Fix $t_{0}\in(0,T)$ and, for $\varrho\in(0,t_{0})$, set
$h_{\varrho}:=\varrho^{-1/p}\mathbf{1}_{[t_{0}-\varrho,t_{0}]}$, so that
$\norm{h_{\varrho}}_{L^{p}(0,T)}=1$. The change of variable $\sigma=t_{0}-s$ gives
\[
\bigl(\IL{\alpha}h_{\varrho}\bigr)(t_{0})
=\frac{\varrho^{-1/p}}{\Gamma(\alpha)}\int_{0}^{\varrho}\sigma^{\alpha-1}d\sigma
=\frac{\varrho^{\,\alpha-1/p}}{\Gamma(\alpha+1)}\xrightarrow[\varrho\to0^{+}]{}\infty,
\]
the exponent being negative exactly because $p<1/\alpha$, while the right-hand side of
\eqref{eq:nobound} stays equal to $C(T,\alpha,p)$.

(ii) Put $p:=r/2$, which satisfies $1\le p<1/\alpha$ precisely when $2\le r<2/\alpha$, and take
$f_{\varrho}:=h_{\varrho}^{1/2}w$ with $w\in\LL$, $\norm{w}_{\LL}=1$. Then
$\norm{f_{\varrho}(t)}^{2}_{\LL}=h_{\varrho}(t)$ and
$\norm{f_{\varrho}}^{2}_{\LLr}=\norm{h_{\varrho}}_{L^{r/2}(0,T)}=1$, so (i) applies
verbatim. For $r=2$ the family is
$f_{\varrho}(x,t)=\varrho^{-1/2}\mathbf{1}_{[t_{0}-\varrho,t_{0}]}(t)w(x)$.
\end{proof}

The operator $\IL{\alpha}$ weights the past of the source
by $(t-s)^{\alpha-1}$, unbounded near $s=t$. Hypothesis \Hthree{} controls total energy, not
distribution in time; the family $f_{\varrho}$ exploits the gap, moving a fixed amount of
energy into an ever shorter interval placed where the weight is largest. For $\alpha=1$ the
weight is constant and the phenomenon disappears; it worsens as $\alpha$ decreases. Hypothesis \Hthreep{} forbids such concentration, and restores the bound.

\begin{proposition}[Exact threshold]\label{prop:threshold}
Let $\alpha\in(0,1)$. Then
\begin{enumerate}[label=\textup{(\roman*)},leftmargin=2.4em]
\item if $p>1/\alpha$ has conjugate $p'$ and $h\in L^{p}(0,T)$ is non-negative, then, for every
$t\in[0,T]$,
\begin{equation}\label{eq:mapping}
\bigl(\IL{\alpha}h\bigr)(t)\le\kappa_{\alpha}(p,T)\norm{h}_{L^{p}(0,T)},
\qquad
\kappa_{\alpha}(p,T):=\frac{1}{\Gamma(\alpha)}
\left(\frac{T^{\,p'(\alpha-1)+1}}{p'(\alpha-1)+1}\right)^{1/p'} ;
\end{equation}
\item under \Hthreep{} the exponent $p:=r/2$ satisfies $p>1/\alpha$, and (i) applies to
$h=\norm{f(\cdot)}^{2}_{\LL}$:
\begin{equation}\label{eq:kappa}
\bigl(\IL{\alpha}\norm{f}^{2}_{\LL}\bigr)(t)
\le\kappa_{\alpha}\!\left(\tfrac r2,T\right)\norm{f}^{2}_{\LLr},
\end{equation}
uniformly in $t\in[0,T]$;
\item if $f\in\LLtwo$ and $h:=\norm{f(\cdot)}^{2}_{\LL}$, then $h\in L^{1}(0,T)$, and $p=1$ does
not satisfy the hypothesis $p>1/\alpha$ of (i) for any $\alpha\in(0,1)$.
\end{enumerate}
\end{proposition}

\begin{proof}
Part (i) is the classical mapping property of the Riemann--Liouville integral into $L^{\infty}$,
\cite[Ch.~1, \S3]{samko}, \cite[Sec.~2.1]{kilbas}; we record the constant because
Theorem~\ref{thm:stability} needs its value. By H\"older with the conjugate exponents $p'$ and
$p$,
\[
\bigl(\IL{\alpha}h\bigr)(t)\le\frac{1}{\Gamma(\alpha)}
\left(\int_{0}^{t}(t-s)^{p'(\alpha-1)}ds\right)^{1/p'}
\left(\int_{0}^{t}h^{p}\right)^{1/p},
\]
and after the change $\sigma=t-s$ the first integral converges precisely when
$p'(\alpha-1)>-1$, in which case it equals $t^{\,p'(\alpha-1)+1}/\bigl(p'(\alpha-1)+1\bigr)$, an
increasing function of $t$ and therefore bounded by its value at $T$. The condition
$p'(\alpha-1)>-1$ reads $p'<1/(1-\alpha)$ and, by conjugation, $p>1/\alpha$.

Parts (ii) and (iii) are the two ways of reading (i) against the hypotheses of
Section~\ref{sec:hyp}, and we do the bookkeeping explicitly, because the exponent that matters is
not the one in which the source is measured but the one in which its squared norm is. We set
$h:=\norm{f(\cdot)}^{2}_{\LL}$.

Under \Hthreep{} we have $f\in\LLr$, that is $\norm{f(\cdot)}_{\LL}\in L^{r}(0,T)$, and
therefore $h\in L^{r/2}(0,T)$ with
\[
\norm{h}_{L^{r/2}(0,T)}
=\left(\int_{0}^{T}\norm{f(t)}^{r}_{\LL}dt\right)^{2/r}
=\norm{f}^{2}_{\LLr}.
\]
The requirement $p>1/\alpha$ of (i) becomes $r/2>1/\alpha$, that is $r>2/\alpha$, which is
\Hthreep{} verbatim. Applying (i), we obtain \eqref{eq:kappa}. The exponent of \Hthreep{} has,
therefore, no origin other than this one step: it is what makes $h$ integrable to a power high
enough for $\IL{\alpha}$ to map it into $L^{\infty}$.

Under \Hthree{} we have $f\in\LLtwo$, so $\norm{f(\cdot)}_{\LL}\in L^{2}(0,T)$ and
$h\in L^{1}(0,T)$ with $\norm{h}_{L^{1}(0,T)}=\norm{f}^{2}_{\LLtwo}$. Moreover $L^{1}$ is all that
\Hthree{} gives, since a function in $L^{2}(0,T)$ need have its square in no $L^{p}$ with $p>1$.
The exponent is thus $p=1$, and $1>1/\alpha$ is false for every $\alpha\in(0,1)$ because
$1/\alpha>1$ there. The failure is not marginal: \Hthree{} misses the admissible range by the
whole interval $[1,1/\alpha]$, which widens as $\alpha$ decreases. This proves (iii).
\end{proof}

\begin{remark}[The excluded range is genuinely excluded]\label{rem:endpoint}
Part (iii) asserts only that the argument of (i) does not apply under \Hthree, which would leave
room for some other argument. There is none, and not merely at $p=1$:
Proposition~\ref{prop:nobound} exhibits, for every $p<1/\alpha$, non-negative $h\in L^{p}(0,T)$
whose fractional averages are unbounded, so no estimate of the form \eqref{eq:mapping} can hold
below the threshold. Propositions~\ref{prop:nobound} and~\ref{prop:threshold} are thus the two
halves of one statement: above $1/\alpha$ the bound holds with an explicit constant, below it no
bound exists.

One exponent is left over, $p=1/\alpha$ itself, and two things make it worth settling. Calling the
threshold exact is otherwise premature, the admissible range being known only to lie between
$(1/\alpha,\infty)$ and $[1/\alpha,\infty)$; and the same exponent returns, after a change of
variable, as the right end of the window of Section~\ref{sec:windowA}, where it is what decides
the case $\alpha=\frac12$. The concentrating family of Proposition~\ref{prop:nobound} cannot
settle it, since its value $\varrho^{\alpha-1/p}$ equals $1$ at $p=1/\alpha$: no single scale
diverges there, however fine it is made. What does settle it is a weight that concentrates at
every scale at once. Fix $t_{0}\in(0,T]$ and put
\begin{equation}\label{eq:endpointh}
h(s):=(t_{0}-s)^{-\alpha}\Bigl(\log\frac{e\,t_{0}}{t_{0}-s}\Bigr)^{-1}\mathbf{1}_{(0,t_{0})}(s),
\qquad s\in(0,T),
\end{equation}
positive and finite on $(0,t_{0})$ because $e\,t_{0}/(t_{0}-s)>e$ there. With $\sigma=t_{0}-s$ and
then $v=\log\bigl(e\,t_{0}/\sigma\bigr)$, the second substitution turning $d\sigma/\sigma$ into
$-dv$ and $\sigma\in(0,t_{0})$ into $v\in(1,\infty)$,
\begin{align*}
\int_{0}^{T}h^{1/\alpha}
&=\int_{0}^{t_{0}}\sigma^{-1}\Bigl(\log\frac{e\,t_{0}}{\sigma}\Bigr)^{-1/\alpha}d\sigma
=\int_{1}^{\infty}v^{-1/\alpha}\,dv=\frac{\alpha}{1-\alpha},\\
\bigl(\IL{\alpha}h\bigr)(t_{0})
&=\frac{1}{\Gamma(\alpha)}\int_{0}^{t_{0}}\sigma^{-1}
\Bigl(\log\frac{e\,t_{0}}{\sigma}\Bigr)^{-1}d\sigma
=\frac{1}{\Gamma(\alpha)}\int_{1}^{\infty}\frac{dv}{v}=\infty ,
\end{align*}
where in the second line the factor $\sigma^{\alpha-1}$ of the kernel has cancelled the
$\sigma^{-\alpha}$ of $h$. The two integrands differ only in the power of the logarithm, and
$1/\alpha>1$ is exactly what separates the convergent case from the divergent one. So
$h\in L^{1/\alpha}(0,T)$ while its fractional average at $t_{0}$ is infinite; and because the
integrands $(t-s)^{\alpha-1}\mathbf{1}_{\{s<t\}}h(s)$ are non-negative and converge pointwise for
a.e. $s\in(0,t_{0})$ as $t\uparrow t_{0}$, Fatou's lemma upgrades this to
$\bigl(\IL{\alpha}h\bigr)(t)\to\infty$. Hence $\IL{\alpha}$ does not map $L^{1/\alpha}(0,T)$ into
$L^{\infty}(0,T)$, estimate \eqref{eq:mapping} fails at the endpoint as well, and the admissible
range is exactly $p>1/\alpha$. Hypothesis \Hthreep{} asks for a strict inequality and nothing
below rests on the endpoint by itself; what it buys is that the threshold is exact rather than
exact up to one exponent, and the function \eqref{eq:endpointh} is used once more in
Section~\ref{sec:windowA}.
\end{remark}

The two cases separate here: estimate \eqref{eq:kappa} is uniform in $t$ and closes an estimate
in $\LinfL$, but it needs \Hthreep, whereas Lemma~\ref{lem:timeint} controls only the integral in
$t$ and holds under \Hthree{} alone, yielding the energy scale $\LHone$. The exponent $2/\alpha$
is not a threshold of the problem: it is the exact threshold of the energy method, which works
with $\norm{u_{n}}^{2}_{\LL}$ and therefore needs $\norm{f}^{2}_{\LL}\in L^{r/2}$ with
$r/2>1/\alpha$. Section~\ref{sec:comparison} returns to this distinction with both cases in hand.

\subsection{Compactness and the passage to the limit}\label{sec:limit}

Four ingredients remain, common to both cases: a bound on the fractional derivative, a
compactness criterion, the identification of the limit, and a proposition that turns uniform
bounds into a weak solution.

\begin{proposition}[Bound in $\HM$]\label{prop:Hminus1}
Assume \Hthree, and suppose that $\norm{u_{n}}^{2}_{\LHone}\le M_{2}$ for every $n$. Then
$\{\partial^{\alpha}_{t}u_{n}\}$ is bounded in $\LHmone$.
\end{proposition}

\begin{proof}
We fix $t$ and take $v\in\HH$ with $\norm{v}_{\HH}\le1$; we write $v=v_{1}+v_{2}$ with $v_{1}$
the orthogonal projection onto $V_{n}$, orthogonal simultaneously in $\LL$ and $\HH$ by
Lemma~\ref{lem:spectral}, so that $\norm{v_{1}}_{\HH}\le1$. Since
$\partial^{\alpha}_{t}u_{n}(t)\in V_{n}$ we have
$\langle\partial^{\alpha}_{t}u_{n}(t),v_{2}\rangle=0$, and \eqref{eq:galerkin} extended by
linearity to $v_{1}$ gives
$\langle\partial^{\alpha}_{t}u_{n}(t),v\rangle=(f(t),v_{1})_{\LL}-a(u_{n}(t),v_{1})$. By
Cauchy--Schwarz with Poincar\'e and by the boundedness of $a$ in \eqref{eq:coercive},
\[
\norm{\partial^{\alpha}_{t}u_{n}(t)}_{\HM}
\le C_{P}\norm{f(t)}_{\LL}+a_{1}\norm{u_{n}(t)}_{\HH},
\]
and squaring, using $(x+y)^{2}\le2x^{2}+2y^{2}$ and integrating, we obtain
$\norm{\partial^{\alpha}_{t}u_{n}}^{2}_{\LHmone}
\le2C_{P}^{2}\norm{f}^{2}_{\LLtwo}+2a_{1}^{2}M_{2}$, a bound independent of $n$.
\end{proof}

The estimate adds no new information about $u_{n}$; it reads it off the equation. We need it
twice, and both times without substitute: it is the second hypothesis of the compactness
criterion below, and it is what places the limit inside the scale $W^{\alpha}_{2,2}$, whose
definition requires $\partial^{\alpha}_{t}u\in\LHmone$.

What we want next is strong convergence of $\{u_{n}\}$ in some $L^{p}\bigl((0,T);\LL\bigr)$. The
classical route is the Aubin--Lions--Simon lemma, and it is closed to us: that lemma asks for a
bound on $\partial_{t}u_{n}$, and we have no such bound, nor can we have one, since the equation
controls only $\partial^{\alpha}_{t}u_{n}$ with $\alpha<1$. The criterion of Li and Liu is built for this situation. It replaces the bound on the integer
derivative by one on the fractional derivative, and charges for the replacement in its first
hypothesis, which bounds the fractional average of $\norm{v}^{p}_{X_{0}}$ instead of the norm in
$L^{p}\bigl((0,T);X_{0}\bigr)$. That difference produces, under \Hthree, a restricted window of
admissible exponents.

\begin{lemma}[Compactness criterion of Li and Liu]\label{lem:liliu}
Let $X_{0}\hookrightarrow\hookrightarrow X_{1}\hookrightarrow X_{2}$ be Banach spaces with the
first embedding compact, let $\gamma\in(0,1)$, $p\in[1,\infty)$, and let
$W\subset L^{1}\bigl((0,T);X_{0}\bigr)$ satisfy
\begin{align}
&\sup_{v\in W}\ \sup_{t\in(0,T)}\ \bigl(\IL{\gamma}\norm{v}^{p}_{X_{0}}\bigr)(t)
\le C_{1},\label{eq:LL1}\\
&\sup_{v\in W}\ \norm{\partial^{\gamma}_{t}v}_{L^{\rho}((0,T);X_{2})}\le C_{2}
\quad\text{for some }\rho\in\Bigl(\tfrac{p}{1+p\gamma},\infty\Bigr)\cap[1,\infty).\label{eq:LL2}
\end{align}
Then $W$ is relatively compact in $L^{p}\bigl((0,T);X_{1}\bigr)$.
\end{lemma}

\begin{proof}
See \cite[Sec.~4]{liliu}.
\end{proof}

We always apply it with
\begin{equation}\label{eq:triple}
X_{0}=\HH,\qquad X_{1}=\LL,\qquad X_{2}=\HM,\qquad\gamma=\alpha,\qquad\rho=2,
\end{equation}
a triple whose first embedding is compact by \Hone. With $\rho=2$, hypothesis \eqref{eq:LL2} is
Proposition~\ref{prop:Hminus1} and the condition $\rho>p/(1+p\alpha)$ holds automatically for
$p\le2$. What has to be checked in each case is \eqref{eq:LL1}, and the exponents for which it can be
checked distinguish the two cases.

\begin{lemma}[Identification of the limit]\label{lem:identification}
Let $p\in[1,\infty)$ and suppose that, along a subsequence,
\begin{align}
u_{n}&\longrightarrow u &&\text{strongly in }L^{p}\bigl((0,T);\LL\bigr),\label{eq:strongconv}\\
u_{n}&\rightharpoonup u &&\text{weakly in }\LHone,\label{eq:weakconv}\\
\partial^{\alpha}_{t}u_{n}&\rightharpoonup\chi &&\text{weakly in }\LHmone .\label{eq:weakconvD}
\end{align}
Then
\begin{enumerate}[label=\textup{(\roman*)},leftmargin=2.4em]
\item the pair $(\chi,g)$ represents the Caputo derivative of $u$ with initial value $g$ in the
sense of Definition~\ref{def:space}; in particular $\chi=\partial^{\alpha}_{t}u$, $u(0)=g$ and
$u\in\WscaleG$;
\item for every $\psi\in C^{\infty}_{0}\bigl((0,T);\HH\bigr)$,
\begin{equation}\label{eq:limitidentity}
\int_{0}^{T}\langle\partial^{\alpha}_{t}u,\psi\rangle\,dt
=\int_{0}^{T}\bigl\langle u,\RLD\psi\bigr\rangle\,dt
-\bigl\langle g,\bigl(\IR{1-\alpha}\psi\bigr)(0)\bigr\rangle .
\end{equation}
\end{enumerate}
\end{lemma}

\begin{proof}
By Lemma~\ref{lem:dens}(ii) the identity \eqref{eq:ibp-n} is available for every $f\in\LLtwo$,
and not only for the regular sources of Lemma~\ref{lem:modes}(ii); we pass to the limit in it
term by term. The left-hand side converges to
$\int_{0}^{T}\langle\chi,\psi\rangle\,dt$ by \eqref{eq:weakconvD}, since $\psi\in\LHone$. In the
first term on the right, $\RLD\psi$ lies in $L^{q}\bigl((0,T);\HH\bigr)$ for every $q$, by
\eqref{eq:rightRLbound}, so the strong convergence \eqref{eq:strongconv} suffices whatever the
value of $p$, even $p=1$. The datum term converges because $u_{n}(0)=g_{n}\to g$ in $\LL$ by
\eqref{eq:gnconv}, while $\bigl(\IR{1-\alpha}\psi\bigr)(0)$ is a fixed element of $\HH$. The
limit of \eqref{eq:ibp-n} is precisely \eqref{eq:spacedef} for the pair $(\chi,g)$, which is
\eqref{eq:limitidentity}; as $\psi$ is arbitrary and $\chi\in\LHmone$ while $u\in\LHone$ by
\eqref{eq:weakconv}, we conclude that $u\in\WscaleG$ with $\partial^{\alpha}_{t}u=\chi$ and
$u(0)=g$.
\end{proof}

\begin{proposition}[From uniform bounds to a weak solution]\label{prop:package}
Assume \Hone, \Htwo, \Hthree, \Hfour, let $\alpha\in(0,1)$, and let $\{u_{n}\}$ be the Galerkin
sequence \eqref{eq:un}--\eqref{eq:gn}. Suppose that
\begin{enumerate}[label=\textup{(\alph*)},leftmargin=2.4em]
\item $\norm{u_{n}}^{2}_{\LHone}\le M_{2}<\infty$ for every $n$, and
\item for some $p\in[1,2]$ hypothesis \eqref{eq:LL1} holds, with the choices \eqref{eq:triple},
uniformly in $n$.
\end{enumerate}
Then there are a subsequence, which we do not relabel, and a function $u\in\WscaleG$ such that
\eqref{eq:strongconv}--\eqref{eq:weakconvD} hold with $\chi=\partial^{\alpha}_{t}u$, and $u$ is a
weak solution of \eqref{eq:P} in the sense of Definition~\ref{def:weak}.
\end{proposition}

\begin{proof}
\emph{Extraction.} By (a) and Proposition~\ref{prop:Hminus1} the sequences $\{u_{n}\}$ and
$\{\partial^{\alpha}_{t}u_{n}\}$ are bounded in the Hilbert spaces $\LHone$ and $\LHmone$, so by
reflexivity \cite[Thm.~3.18, Cor.~3.30]{brezis} a subsequence satisfies \eqref{eq:weakconv} and
\eqref{eq:weakconvD}. By (a), (b) and Proposition~\ref{prop:Hminus1} both hypotheses of
Lemma~\ref{lem:liliu} hold, so $\{u_{n}\}$ is relatively compact in
$L^{p}\bigl((0,T);\LL\bigr)$ and a further subsequence satisfies \eqref{eq:strongconv}; the
strong and the weak limit agree, both being limits in $\mathcal{D}'\bigl((0,T);\LL\bigr)$.
Lemma~\ref{lem:identification} then gives $\chi=\partial^{\alpha}_{t}u$, $u(0)=g$ and
$u\in\WscaleG$.

\emph{Projection of the test function.} We fix $\phi\in C^{\infty}_{0}\bigl((0,T);\HH\bigr)$ and
set
\[
\phi_{n}(t):=\sum_{j\le n}(\phi(t),w_{j})_{\LL}w_{j}.
\]
For each $t$ this is the orthogonal
projection of $\phi(t)$ onto $V_{n}$ in $\LL$ and in $\HH$ alike, by Lemma~\ref{lem:spectral}, so
$\norm{\phi_{n}(t)-\phi(t)}\to0$ in both norms while $\norm{\phi_{n}(t)}\le\norm{\phi(t)}$; since
$\phi$ is continuous with compact support, dominated convergence gives
\begin{equation}\label{eq:phiconv}
\phi_{n}\longrightarrow\phi\quad\text{strongly in }\LHone\text{ and in }\LLtwo .
\end{equation}

\emph{Passage to the limit.} Multiplying \eqref{eq:galerkin} by the coefficients of $\phi_{n}$,
summing and integrating, we obtain
\begin{equation}\label{eq:galerkinweak}
\int_{0}^{T}\langle\partial^{\alpha}_{t}u_{n},\phi_{n}\rangle\,dt
+\int_{0}^{T}a(u_{n},\phi_{n})\,dt=\int_{0}^{T}(f,\phi_{n})_{\LL}\,dt .
\end{equation}
We split each of the first two terms as $\phi_{n}=(\phi_{n}-\phi)+\phi$. The pieces carrying
$\phi_{n}-\phi$ vanish in the limit: by Cauchy--Schwarz they are bounded by
$\norm{\partial^{\alpha}_{t}u_{n}}_{\LHmone}\norm{\phi_{n}-\phi}_{\LHone}$ and by
$a_{1}\norm{u_{n}}_{\LHone}\norm{\phi_{n}-\phi}_{\LHone}$ respectively, the
first factors being bounded by Proposition~\ref{prop:Hminus1} and (a) and the second tending to
zero by \eqref{eq:phiconv}. The pieces carrying $\phi$ converge by \eqref{eq:weakconvD} together
with Lemma~\ref{lem:identification}, and by \eqref{eq:weakconv}, since
$v\mapsto\int_{0}^{T}a(v,\phi)dt$ is a continuous linear functional on $\LHone$. The source term
converges by \eqref{eq:phiconv} and Cauchy--Schwarz, using \Hthree. Hence
\begin{equation}\label{eq:limitequation}
\int_{0}^{T}\langle\partial^{\alpha}_{t}u,\phi\rangle\,dt+\int_{0}^{T}a(u,\phi)\,dt
=\int_{0}^{T}(f,\phi)_{\LL}\,dt
\end{equation}
for every test function. Substituting \eqref{eq:limitidentity} for the first term, which is where
the initial datum reappears, and moving the datum to the right, we obtain exactly
\eqref{eq:weak}. Hence $u$ is a weak solution, with the prescribed initial datum.
\end{proof}

\subsection{Uniqueness}\label{sec:uniqueness}

Uniqueness needs neither of the two alternatives in its stronger form: it holds under \Hthree{}
alone and for every order of the subdiffusive range. The result is not new, see
\cite{zacher09,zacher19,kubica,sakamoto}, but we give the proof, because our notion of weak
solution is stated through the pair of Definition~\ref{def:space} and not in the form those
references use, and because the same spectral projection reappears in
Section~\ref{sec:spectral}. One scalar lemma is all the fractional calculus the argument needs.

\begin{lemma}[Scalar fundamental theorem]\label{lem:ftfcscalar}
Let $\alpha\in(0,1)$ and let $b,\chi\in L^{2}(0,T)$ satisfy
\begin{equation}\label{eq:scalarweak}
\int_{0}^{T}b(t)\,\bigl(\RLD\zeta\bigr)(t)\,dt=\int_{0}^{T}\chi(t)\,\zeta(t)\,dt
\qquad\forall\,\zeta\in C^{\infty}_{0}(0,T).
\end{equation}
Then $\bigl(\IL{\alpha}\chi\bigr)(t)=b(t)$ for a.e. $t\in(0,T)$.
\end{lemma}

\begin{proof}
By Lemma~\ref{lem:RLprops}(iv) with $\gamma=1-\alpha$, and by \eqref{eq:RLderright},
identity \eqref{eq:scalarweak} says that the function $\IL{1-\alpha}b$, which lies in
$L^{2}(0,T)$ by Lemma~\ref{lem:RLprops}(i), has distributional derivative $\chi\in L^{2}(0,T)$ on
$(0,T)$; hence $\IL{1-\alpha}b\in AC([0,T])$ with $\DL b=\chi$ a.e., and testing
\eqref{eq:scalarweak} with $\zeta$ supported near $0$ gives
$\bigl(\IL{1-\alpha}b\bigr)(0)=0$. Applying $\IL{\alpha}$ and using the composition formula
\[
\bigl(\IL{\alpha}\DL b\bigr)(t)=b(t)
-\bigl(\IL{1-\alpha}b\bigr)(0)\,\frac{t^{\alpha-1}}{\Gamma(\alpha)}
\]
of \cite[Sec.~2.1]{kilbas}, valid for $\IL{1-\alpha}b\in AC([0,T])$, we obtain
$\bigl(\IL{\alpha}\chi\bigr)=b$ a.e.
\end{proof}

\begin{theorem}[Uniqueness]\label{thm:uniqueness}
Assume \Hone, \Htwo, \Hthree{} and \Hfour, and let $\alpha\in(0,1)$. Then \eqref{eq:P} has at most one weak solution in the sense of Definition~\ref{def:weak}.
\end{theorem}

\begin{proof}
Let $U_{1},U_{2}$ be weak solutions with the same data and $W:=U_{1}-U_{2}$. Subtracting the two
identities \eqref{eq:weak}, which are linear in $u$, every data term cancels and
\begin{equation}\label{eq:diffweak}
\int_{0}^{T}\bigl\langle W,\RLD\phi\bigr\rangle\,dt+\int_{0}^{T}a(W,\phi)\,dt=0
\end{equation}
for every test function. We fix $j$ and take $\phi(t)=\zeta(t)w_{j}$ with
$\zeta\in C^{\infty}_{0}(0,T)$ arbitrary, which is admissible since $w_{j}\in\HH$, and we write
$b_{j}:=(W(\cdot),w_{j})_{\LL}\in L^{2}(0,T)$. Since $\RLD(\zeta w_{j})=(\RLD\zeta)w_{j}$ and,
by \eqref{eq:eigen} and the symmetry of $a$, $a(W,w_{j})=\lambda_{j}b_{j}$, identity
\eqref{eq:diffweak} becomes
\[
\int_{0}^{T}b_{j}(t)\,\bigl(\RLD\zeta\bigr)(t)\,dt
=\int_{0}^{T}\bigl(-\lambda_{j}b_{j}(t)\bigr)\,\zeta(t)\,dt ,
\]
which is \eqref{eq:scalarweak} with $\chi=-\lambda_{j}b_{j}\in L^{2}(0,T)$. Lemma~\ref{lem:ftfcscalar} gives $b_{j}+\lambda_{j}\IL{\alpha}b_{j}=0$, which is
\eqref{eq:volterra} with $F=0$ and $\eta_{0}=0$, so $b_{j}=0$ a.e. by
Lemma~\ref{lem:scalar}(ii). This holds for every $j$, and $\{w_{j}\}$ is an orthonormal basis of
$\LL$, so $W(t)=0$ for a.e. $t$.

The argument uses the weak formulation alone and never the pair of
Definition~\ref{def:space}; the ambiguity of that pair for $\alpha<\frac12$, recorded in
Remark~\ref{rem:pairunique}, therefore does not affect the conclusion.
\end{proof}

\begin{corollary}[Linearity of the solution operator]\label{cor:linear}
Assume the hypotheses of Theorem~\ref{thm:A} or of Theorem~\ref{thm:B}. Then
\begin{enumerate}[label=\textup{(\roman*)},leftmargin=2.4em]
\item the map sending the pair of data $(g,f)$ to the weak solution of \eqref{eq:P} is well
defined and linear;
\item if an a priori bound $\norm{u}_{X}\le\Phi(g,f)$ holds for the constructed solution and for
every admissible pair of data, then the same bound holds for the weak solution;
\item if, in addition, $u_{1}$ and $u_{2}$ are the weak solutions for the data $(g_{1},f_{1})$
and $(g_{2},f_{2})$, then $\norm{u_{1}-u_{2}}_{X}\le\Phi(g_{1}-g_{2},\,f_{1}-f_{2})$.
\end{enumerate}
\end{corollary}

\begin{proof}
Well defined by existence and Theorem~\ref{thm:uniqueness}. If $u^{(1)},u^{(2)}$ are the
solutions for two pairs of data and $c_{1},c_{2}\in\R$, then $c_{1}u^{(1)}+c_{2}u^{(2)}$
satisfies \eqref{eq:weak} for the combined data, those identities being linear in
$(u;\text{data})$; by uniqueness it is the solution for that pair. The last assertion is
immediate, the constructed solution being the weak solution.
\end{proof}

\subsection{Modes of the approximations and the spectral bound}\label{sec:spectral}

The system \eqref{eq:modes} is decoupled, so each approximation can be written out mode by mode
with explicit kernels. Nothing here uses the existence of a limit, and the bound this yields is
what Sections~\ref{sec:caseA}, \ref{sec:caseB} and~\ref{sec:sharp} draw on.

\begin{lemma}[Modes of the approximations]\label{lem:modesofu}
Let $\alpha\in(0,1)$, assume \Hone, \Htwo{} and \Hfour, and let
$f\in L^{1}\bigl((0,T);\LL\bigr)$. Then
\begin{enumerate}[label=\textup{(\roman*)},leftmargin=2.4em]
\item the Galerkin approximations \eqref{eq:un} are the partial sums
$u_{n}(t)=\sum_{j\le n}c_{j}(t)w_{j}$ of one fixed expansion, with
\begin{equation}\label{eq:modesofu}
\begin{gathered}
c_{j}(t)=(g,w_{j})_{\LL}E_{\alpha,1}\bigl(-\lambda_{j}t^{\alpha}\bigr)+(k_{j}*F_{j})(t),\\
k_{j}(\sigma):=\sigma^{\alpha-1}E_{\alpha,\alpha}\bigl(-\lambda_{j}\sigma^{\alpha}\bigr),
\qquad F_{j}:=(f(\cdot),w_{j})_{\LL} ;
\end{gathered}
\end{equation}
\item for every $j\ge1$ and every $\sigma\in(0,T]$,
\begin{equation}\label{eq:kernelbounds}
\theta_{j}\,\sigma^{\alpha-1}\le k_{j}(\sigma)\le\frac{\sigma^{\alpha-1}}{\Gamma(\alpha)},
\qquad \theta_{j}:=E_{\alpha,\alpha}\bigl(-\lambda_{j}T^{\alpha}\bigr)>0 .
\end{equation}
\end{enumerate}
\end{lemma}

\begin{proof}
The system \eqref{eq:modes} is decoupled and neither it nor the datum \eqref{eq:gn} involves $n$,
so $c_{j,n}=c_{j}$ for every $n\ge j$, and Lemma~\ref{lem:scalar}(i)--(ii) identifies $c_{j}$ as
\eqref{eq:modesofu}. The two-sided bound \eqref{eq:kernelbounds} is \eqref{eq:MLbound} together
with the fact that $x\mapsto E_{\alpha,\alpha}(-x)$ is non-increasing, so that its value at
$\lambda_{j}\sigma^{\alpha}$ lies between its value at $\lambda_{j}T^{\alpha}$ and
$1/\Gamma(\alpha)$.
\end{proof}

\begin{proposition}[Spectral bound on the approximations]\label{prop:spectralbound}
Let $\alpha\in(0,1)$, assume \Hone, \Htwo{} and \Hfour, and let $q>1/\alpha$ be an exponent with
$f\in\LLq$. Then
\begin{enumerate}[label=\textup{(\roman*)},leftmargin=2.4em]
\item for every $n\in\mathbb{N}$,
\begin{equation}\label{eq:spectralbound}
\sup_{t\in[0,T]}\norm{u_{n}(t)}_{\LL}\le\norm{g}_{\LL}+\kappa_{\alpha}(q,T)\norm{f}_{\LLq},
\end{equation}
with $\kappa_{\alpha}(q,T)$ the constant of \eqref{eq:mapping};
\item if a subsequence of $\{u_{n}\}$ converges weakly in $\LHone$ to a weak solution $u$ of
\eqref{eq:P}, then $u\in\LinfL$ and \eqref{eq:spectralbound} holds with $u$ in place of $u_{n}$.
\end{enumerate}
\end{proposition}

\begin{proof}
(i) By Lemma~\ref{lem:modesofu}(i) and the triangle inequality in $\ell^{2}$,
\[
\norm{u_{n}(t)}_{\LL}\le\Bigl(\sum_{j\le n}\abs{(g,w_{j})_{\LL}}^{2}
E_{\alpha,1}\bigl(-\lambda_{j}t^{\alpha}\bigr)^{2}\Bigr)^{1/2}
+\Bigl(\sum_{j\le n}\abs{(k_{j}*F_{j})(t)}^{2}\Bigr)^{1/2}.
\]
The first sum is at most $\norm{g}_{\LL}$, because $0<E_{\alpha,1}(-x)\le1$ by
\eqref{eq:MLbound} and $\{w_{j}\}$ is orthonormal. For the second, the upper bound in
\eqref{eq:kernelbounds} is uniform in $j$, so the integral Minkowski inequality in $\ell^{2}$ and
Bessel's inequality give
\[
\Bigl(\sum_{j\le n}\abs{(k_{j}*F_{j})(t)}^{2}\Bigr)^{1/2}
\le\int_{0}^{t}\frac{(t-s)^{\alpha-1}}{\Gamma(\alpha)}
\Bigl(\sum_{j\le n}\abs{F_{j}(s)}^{2}\Bigr)^{1/2}ds
\le\bigl(\IL{\alpha}\norm{f(\cdot)}_{\LL}\bigr)(t),
\]
and Proposition~\ref{prop:threshold}(i) with $p=q$ bounds the right-hand side by
$\kappa_{\alpha}(q,T)\norm{f}_{\LLq}$, uniformly in $t$ and in $n$.

(ii) Bound (i) is uniform in $n$, and $\LinfL$ is the dual of $L^{1}\bigl((0,T);\LL\bigr)$, so
bounded sets there are weak-$*$ sequentially compact: a further subsequence converges weakly-$*$
in $\LinfL$, and its limit is $u$ by uniqueness of limits in
$\mathcal{D}'\bigl((0,T);\LL\bigr)$. Lower semicontinuity of the norm under weak-$*$ convergence
carries the bound to $u$.
\end{proof}

Proposition~\ref{prop:threshold} showed the exponent $r>2/\alpha$ to be what the energy argument
consumes. Bound \eqref{eq:spectralbound} needs only $q>1/\alpha$, half of that exponent, and the
gap between the two is the cost of the energy route. One consequence is worth stating at once,
because it is what keeps the comparison of Section~\ref{sec:comparison} honest: with $q=2$,
admissible exactly when $\alpha>\frac12$, the bound holds under \Hthree{} alone. The uniform
estimate in $\LinfL$ is therefore not what \Hthreep{} buys; what \Hthreep{} buys is the range
$\alpha\le\frac12$, where by Proposition~\ref{prop:sharp} the estimate is false under \Hthree.

\section{The case \texorpdfstring{\Hthree}{(H3)}: the weaker hypothesis}\label{sec:caseA}

\subsection{Statement}

\begin{theorem}[Existence under \Hthree]\label{thm:A}
Let $\alpha\in\bigl(\tfrac12,1\bigr)$ and assume \Hone, \Htwo, \Hthree{} and \Hfour. Then
\begin{enumerate}[label=\textup{(\roman*)},leftmargin=2.4em]
\item problem \eqref{eq:P} has a weak solution $u\in\WscaleG$;
\item the Galerkin approximations \eqref{eq:un} converge to $u$, the whole sequence, strongly in
$L^{p}\bigl((0,T);\LL\bigr)$ for every $p\in[1,2\alpha)$, weakly in $\LHone$ and weakly-$*$ in
$\LinfL$;
\item the solution satisfies $u\in\LinfL$, with
\begin{equation}\label{eq:LinfA}
\norm{u}_{\LinfL}\le\norm{g}_{\LL}+\kappa_{\alpha}(2,T)\norm{f}_{\LLtwo} .
\end{equation}
\end{enumerate}
\end{theorem}

The solution is unique by Theorem~\ref{thm:uniqueness}.

Two numbers occur in (ii). Both come from Lemma~\ref{lem:liliu}. We must verify the
first hypothesis of that lemma with the only bound available in the present case, namely the
energy bound \eqref{eq:M2A}, and we can do so for some exponents and not for others. The
admissible set is exactly $[1,2\alpha)$,
which accounts for the range of $p$, and it is non-empty precisely when $\alpha>\frac12$, which
accounts for the restriction on $\alpha$. The restriction $\alpha>\frac12$ thus records the
non-emptiness of an interval. Section~\ref{sec:consequencesA} takes up what it restricts, and
Section~\ref{sec:sharp} settles it.

Conclusion (iii) is Proposition~\ref{prop:spectralbound} with $q=2$, admissible exactly because
$\alpha>\frac12$. It is worth separating at once what gives it from what does not. The energy
route of this section yields no bound in $\LinfL$, for the reason
Proposition~\ref{prop:nobound} makes precise, and the mode representation of
Section~\ref{sec:spectral} supplies one notwithstanding. Hypothesis \Hthree{} is therefore not
what costs the uniform estimate; the order is, and only below $\frac12$, as
Proposition~\ref{prop:sharp} shows.

The a priori estimates of Section~\ref{sec:boundsA} hold for every $\alpha\in(0,1)$. The
restriction $\alpha>\frac12$ enters only in Section~\ref{sec:windowA}, through the compactness
step.

\subsection{A priori estimates and uniform bounds}\label{sec:boundsA}

By Proposition~\ref{prop:nobound} the term $\bigl(\IL{\alpha}\norm{f}^{2}_{\LL}\bigr)(t)$ in
\eqref{eq:energyL2} admits no bound uniform in $t$ under \Hthree. A pointwise estimate in time is
therefore out of reach along this route. What survives is Lemma~\ref{lem:timeint}, which controls the integral in $t$ of a fractional
integral: we integrate in time first and apply Gr\"onwall afterwards. The unknown changes on the
way. The inequality ceases to be about $\norm{u_{n}(t)}^{2}_{\LL}$ and becomes about its
primitive, and with that goes any bound in $\LinfL$ along this route. What the route yields
instead is the energy scale.

\begin{proposition}[Uniform bounds under \Hthree]\label{prop:boundsA}
Let $\alpha\in(0,1)$ and assume \Hone, \Htwo, \Hthree{} and \Hfour. Then there are constants
$N_{1},M_{2}>0$, independent of $n$, such that
\begin{equation}\label{eq:N1}
\int_{0}^{T}\norm{u_{n}(t)}^{2}_{\LL}dt\le N_{1}:=D(T)\,E_{\alpha}\bigl(T^{\alpha}\bigr),
\qquad
D(\tau):=\tau\norm{g}^{2}_{\LL}+\frac{\tau^{\alpha}\,
\norm{f}^{2}_{\LLtwo}}{\Gamma(\alpha+1)},
\end{equation}
\begin{equation}\label{eq:M2A}
\int_{0}^{T}\norm{u_{n}(t)}^{2}_{\HH}dt\le M_{2}
\end{equation}
hold for every $n\in\mathbb{N}$.
\end{proposition}

\begin{proof}
Throughout we take $\varepsilon=1$ in \eqref{eq:energyL2} and \eqref{eq:energyH1}, which suffices
here.

\emph{Step 1: integration in time.} We integrate \eqref{eq:energyL2} over $(0,\tau)$ with
$\tau\in(0,T]$ arbitrary, and write $Y_{n}(\tau):=\int_{0}^{\tau}\norm{u_{n}(t)}^{2}_{\LL}dt$ for
the primitive of $\norm{u_{n}(\cdot)}^{2}_{\LL}$, a continuous non-negative function. The
left-hand side is $Y_{n}(\tau)$ by definition. The constant term contributes
$\tau\norm{g_{n}}^{2}_{\LL}\le\tau\norm{g}^{2}_{\LL}$ by \eqref{eq:gnconv}. For the source term,
Lemma~\ref{lem:timeint} with $\gamma=\alpha$ and $h=\norm{f(\cdot)}^{2}_{\LL}\in L^{1}(0,T)$,
where we use \Hthree{} alone, gives
\[
\int_{0}^{\tau}\bigl(\IL{\alpha}\norm{f}^{2}_{\LL}\bigr)(t)\,dt
\le\frac{\tau^{\alpha}}{\Gamma(\alpha+1)}\norm{f}^{2}_{\LLtwo}.
\]
For the term in the unknown, the same lemma with $h=\norm{u_{n}(\cdot)}^{2}_{\LL}$, this time in
the form of its second equality, gives
$\int_{0}^{\tau}\bigl(\IL{\alpha}\norm{u_{n}}^{2}_{\LL}\bigr)(t)\,dt
=\bigl(\IL{\alpha}Y_{n}\bigr)(\tau)$, so that
the Gr\"onwall structure survives the change of unknown. Adding the four contributions, we
obtain
\begin{equation}\label{eq:gronwallY}
Y_{n}(\tau)\le D(\tau)+\bigl(\IL{\alpha}Y_{n}\bigr)(\tau),
\end{equation}
with $D$ as in \eqref{eq:N1}.

\emph{Step 2: Gr\"onwall.} Inequality \eqref{eq:gronwallY} is \eqref{eq:gronwall-hyp} with
$y=Y_{n}$, $\gamma=\alpha$, $\varepsilon=1$ and $a=D$, and $D$ is non-decreasing because both
$\tau$ and $\tau^{\alpha}$ are. Lemma~\ref{lem:gronwall} gives
$Y_{n}(\tau)\le D(\tau)E_{\alpha}(\tau^{\alpha})$, and $\tau=T$ is \eqref{eq:N1}.

\emph{Step 3: the bound in $\HH$, and why a detour is needed.} The natural attempt is to repeat
Steps 1 and 2 on the second inequality \eqref{eq:energyH1}. We take $\varepsilon=1$ there and
integrate over $(0,T)$. On the left, the first equality of \eqref{eq:timeint} with $\gamma=\alpha$
and $h=\norm{u_{n}(\cdot)}^{2}_{\HH}$ turns the integral of the fractional integral into a single
weighted integral,
\begin{equation}\label{eq:weightedT}
2a_{0}\int_{0}^{T}\bigl(\IL{\alpha}\norm{u_{n}}^{2}_{\HH}\bigr)(t)\,dt
=\frac{2a_{0}}{\Gamma(\alpha+1)}\int_{0}^{T}\norm{u_{n}(s)}^{2}_{\HH}\,(T-s)^{\alpha}ds ,
\end{equation}
while the right-hand side is bounded exactly as in Steps 1 and 2, by
$T\norm{g}^{2}_{\LL}+\Gamma(\alpha+1)^{-1}T^{\alpha}\bigl(\norm{f}^{2}_{\LLtwo}+N_{1}\bigr)$.

The weight $(T-s)^{\alpha}$ vanishes at $s=T$, so
\eqref{eq:weightedT} gives no information on the size of $\norm{u_{n}(s)}_{\HH}$ for $s$ near
$T$; a sequence concentrating there can satisfy it with $\int_{0}^{T}\norm{u_{n}}^{2}_{\HH}$
arbitrarily large. Nor can the weight be removed by manipulating \eqref{eq:weightedT}, since
Lemma~\ref{lem:timeint} produces it as an identity.

The weight vanishes at the right endpoint of the interval of integration, which is ours to
choose. If we integrate up to $2T$ instead of up to $T$, the degeneracy occurs at $s=2T$ and the
weight remains bounded below on $(0,T)$.

For this we need a problem posed on $[0,2T]$ whose restriction to $[0,T]$ is the one we started
with. We set
\[
\tilde f:=f\,\mathbf{1}_{(0,T)},
\]
which belongs to $\mathcal{D}$ over $(0,2T)$ because $\operatorname{supp}f$ is compact in
$(0,T)$, so that the extension by zero is still smooth, and let $\tilde u_{n}$ be the solution of
the Galerkin system \eqref{eq:modes} on $[0,2T]$ with source $\tilde f$ and the same initial
datum $g_{n}$. Two observations make this legitimate. First, the hypotheses \Hone, \Htwo{} and
\Hfour{} do not involve the horizon, and $\tilde f$ satisfies \Hthree{} on $(0,2T)$ with the
same norm,
\begin{equation}\label{eq:sametilde}
\norm{\tilde f}_{L^{2}((0,2T);\LL)}=\norm{f}_{\LLtwo},
\end{equation}
since the extension is by zero; so every estimate proved so far holds verbatim on $[0,2T]$ for
$\tilde u_{n}$, with $T$ replaced by $2T$ and with the \emph{same} norm of the source. Second,
\begin{equation}\label{eq:restriction}
\tilde u_{n}=u_{n}\quad\text{on }[0,T].
\end{equation}

Indeed, let us fix $t\in(0,T)$. By the causality recorded in Lemma~\ref{lem:RLprops}(iii) the
value $\bigl(\IL{\alpha}v\bigr)(t)$ depends only on $v|_{(0,t)}$, whichever interval $v$ is
defined on: the integration in \eqref{eq:RLleft} runs over $(0,t)$ and over nothing else. Hence
$\bigl(\IL{\alpha}\tilde u_{n}\bigr)(t)$ and $\bigl(\IL{\alpha}\tilde f\bigr)(t)$ are computed
from the restrictions of $\tilde u_{n}$ and $\tilde f$ to $(0,t)\subset(0,T)$, where
$\tilde f=f$. Therefore $\tilde u_{n}|_{[0,T]}$ satisfies, on $(0,T)$, the same Volterra
equation \eqref{eq:volterra} as $u_{n}$, with the same source and the same initial datum; the
uniqueness in Lemma~\ref{lem:scalar}(ii), applied mode by mode, gives \eqref{eq:restriction}. The
auxiliary problem agrees with the original on $[0,T]$.

We now run Steps 1 and 2 on $[0,2T]$. Integrating \eqref{eq:energyH1} with $\varepsilon=1$,
written for $\tilde u_{n}$ on $[0,2T]$, over $(0,2T)$, and using \eqref{eq:timeint} on the left
and \eqref{eq:sametilde} on the right, we obtain
\begin{equation}\label{eq:estimate2T}
\begin{gathered}
\frac{2a_{0}}{\Gamma(\alpha+1)}\int_{0}^{2T}\norm{\tilde u_{n}(s)}^{2}_{\HH}(2T-s)^{\alpha}ds
\ \le\ B,\\[4pt]
B:=2T\norm{g}^{2}_{\LL}+\frac{(2T)^{\alpha}}{\Gamma(\alpha+1)}
\Bigl(\norm{f}^{2}_{\LLtwo}+\widetilde N_{1}\Bigr),
\end{gathered}
\end{equation}
where $\widetilde N_{1}:=D(2T)E_{\alpha}\bigl((2T)^{\alpha}\bigr)$ is the constant \eqref{eq:N1}
computed on $[0,2T]$, with the same $\norm{f}_{\LLtwo}$, again by \eqref{eq:sametilde}. The
constant $B$ depends on $T$, $\alpha$, $a_{0}$ and the data, and not on $n$.

It remains to read \eqref{eq:estimate2T} back on $(0,T)$, in four steps:
\begin{align*}
\int_{0}^{T}\norm{u_{n}(s)}^{2}_{\HH}ds
&=\int_{0}^{T}\norm{\tilde u_{n}(s)}^{2}_{\HH}ds
&&\text{by \eqref{eq:restriction},}\\[2pt]
&\le T^{-\alpha}\!\int_{0}^{T}\norm{\tilde u_{n}(s)}^{2}_{\HH}(2T-s)^{\alpha}ds
&&\text{as }2T-s\ge T\text{ on }(0,T),\\[2pt]
&\le T^{-\alpha}\!\int_{0}^{2T}\norm{\tilde u_{n}(s)}^{2}_{\HH}(2T-s)^{\alpha}ds
&&\text{integrand }\ge0,\\[2pt]
&\le\frac{\Gamma(\alpha+1)}{2a_{0}T^{\alpha}}\,B
&&\text{by \eqref{eq:estimate2T}.}
\end{align*}
The second step is what the detour is for: on $(0,T)$ the weight $(2T-s)^{\alpha}$ is
bounded below by $T^{\alpha}>0$, so that we may divide by it, whereas we may not divide by
$(T-s)^{\alpha}$ in \eqref{eq:weightedT}. This proves \eqref{eq:M2A} with the explicit constant
\[
M_{2}:=\frac{\Gamma(\alpha+1)}{2a_{0}T^{\alpha}}
\left[2T\norm{g}^{2}_{\LL}+\frac{(2T)^{\alpha}}{\Gamma(\alpha+1)}
\Bigl(\norm{f}^{2}_{\LLtwo}+D(2T)E_{\alpha}\bigl((2T)^{\alpha}\bigr)\Bigr)\right],
\]
independent of $n$.
\end{proof}

\begin{remark}[What the detour costs]\label{rem:detour}
The auxiliary horizon is only a device: the extended source carries no energy beyond
$T$, and \eqref{eq:restriction} says that we have left the approximations on $[0,T]$ untouched.
What we pay for it is visible in $M_{2}$, which involves $\norm{f}_{\LLtwo}$ and the datum
through $D(2T)E_{\alpha}((2T)^{\alpha})$ rather than $D(T)E_{\alpha}(T^{\alpha})$: the constant
is the one the interval $[0,2T]$ would produce, and it degenerates like $T^{-\alpha}$ as
$T\downarrow0$. Any factor $\lambda>1$ in place of $2$ works equally well, with $T^{\alpha}$
replaced by $((\lambda-1)T)^{\alpha}$, so that there is a trade-off between the two constants;
nothing in what follows depends on the choice, and we take $\lambda=2$ for definiteness. The
constant is not optimal, and we have not tried to optimise it.
\end{remark}

\subsection{The window of admissible exponents}\label{sec:windowA}

\begin{proposition}[Verification of \eqref{eq:LL1} under \Hthree]\label{prop:compactA}
Assume \Hthree{} and take the choices \eqref{eq:triple}. Then
\begin{enumerate}[label=\textup{(\roman*)},leftmargin=2.4em]
\item hypothesis \eqref{eq:LL1} holds for every $p$ with $1\le p<2\alpha$;
\item if \eqref{eq:M2A} is the only information available on $\norm{u_{n}}_{\HH}$, it fails for
every $p$ with $2\alpha\le p<2$;
\item such a $p$ as in (i) exists if and only if $\alpha>\frac12$.
\end{enumerate}
\end{proposition}

\begin{proof}
(i) We set $\Psi_{n}(s):=\norm{u_{n}(s)}_{\HH}$. By \eqref{eq:M2A}, $\Psi_{n}\in L^{2}(0,T)$ with
$\norm{\Psi_{n}}^{2}_{L^{2}}\le M_{2}$ uniformly in $n$, and that energy bound is all we know
about $\Psi_{n}$ in this case; the threshold comes out of it. Let $1\le p<2$, so that
$\Psi_{n}^{p}\in L^{2/p}(0,T)$ with
$\norm{\Psi^{p}_{n}}_{L^{2/p}}=\norm{\Psi_{n}}^{p}_{L^{2}}$, and let us apply H\"older on $(0,t)$
with the conjugate exponents $q=2/(2-p)$ and $q'=2/p$:
\[
\int_{0}^{t}(t-s)^{\alpha-1}\Psi_{n}(s)^{p}ds
\le\norm{(t-\cdot)^{\alpha-1}}_{L^{q}(0,t)}\norm{\Psi^{p}_{n}}_{L^{q'}(0,t)} .
\]
After the change $\sigma=t-s$ the first norm is finite exactly when $q(\alpha-1)>-1$, that is
\[
\frac{2}{2-p}(\alpha-1)>-1
\iff 2(\alpha-1)>-(2-p)
\iff 2\alpha-2>p-2
\iff\boxed{\,p<2\alpha\,},
\]
the first equivalence using $2-p>0$; and in that case it equals
$\bigl(t^{q(\alpha-1)+1}/(q(\alpha-1)+1)\bigr)^{1/q}$, increasing in $t$ and hence bounded on
$(0,T)$ by its value at $T$. The second norm is at most $M_{2}^{p/2}$. This is \eqref{eq:LL1}.

(ii) Let $2\alpha<p<2$ and let $C_{1}>0$ be given. Fix $t_{0}\in(0,T)$ and, for
$\varrho\in(0,t_{0})$, set
$\Psi_{\varrho}:=M_{2}^{1/2}\varrho^{-1/2}\mathbf{1}_{[t_{0}-\varrho,t_{0}]}$, so
that $\norm{\Psi_{\varrho}}^{2}_{L^{2}(0,T)}=M_{2}$, which is everything \eqref{eq:M2A}
asserts about $\norm{u_{n}(\cdot)}_{\HH}$. With the change $\sigma=t_{0}-s$,
\[
\bigl(\IL{\alpha}\Psi_{\varrho}^{p}\bigr)(t_{0})
=\frac{M_{2}^{p/2}}{\Gamma(\alpha+1)}\,\varrho^{\,\alpha-p/2}
\xrightarrow[\varrho\to0^{+}]{}\infty ,
\]
the exponent being negative because $p>2\alpha$. No bound \eqref{eq:LL1} therefore follows from
\eqref{eq:M2A} alone for such $p$. At the endpoint $p=2\alpha$ that exponent vanishes and the
family above stays bounded, so we appeal to Remark~\ref{rem:endpoint} instead. Writing $q:=2/p$,
hypothesis \eqref{eq:LL1} at the exponent $p$ is \eqref{eq:mapping} applied to
$\Psi^{p}_{n}\in L^{q}(0,T)$, and the admissibility condition $q>1/\alpha$ is $p<2\alpha$: the
endpoint of the window is the endpoint of the threshold. Explicitly, fix $t_{0}\in(0,T)$, let $h$
be as in \eqref{eq:endpointh} and set $c:=M_{2}(1-\alpha)/\alpha$ and
\[
\Psi:=c^{1/2}h^{1/(2\alpha)}
=c^{1/2}(t_{0}-s)^{-1/2}\Bigl(\log\frac{e\,t_{0}}{t_{0}-s}\Bigr)^{-1/(2\alpha)} .
\]
Then $\Psi^{2}=c\,h^{1/\alpha}$, so $\norm{\Psi}^{2}_{L^{2}(0,T)}=M_{2}$ by
Remark~\ref{rem:endpoint}, which is again everything \eqref{eq:M2A} asserts, while
$\Psi^{2\alpha}=c^{\alpha}h$ and therefore
$\bigl(\IL{\alpha}\Psi^{2\alpha}\bigr)(t)\to\infty$ as $t\uparrow t_{0}$. The exponent
$p=2\alpha$ is inadmissible as well, and (ii) holds on all of $[2\alpha,2)$.

(iii) Lemma~\ref{lem:liliu} needs $p\ge1$ for $L^{p}\bigl((0,T);X_{1}\bigr)$ to make sense, so by
(i) and (ii) the admissible exponents form exactly the interval $[1,2\alpha)$. The endpoint of
(ii) is what decides the case $\alpha=\frac12$: there $2\alpha=1$, and $p=1$ is the only exponent
the lemma would admit at all, so without it the interval could not be asserted to be empty.
Hence
\[
[1,2\alpha)\neq\emptyset\iff2\alpha>1\iff\alpha>\tfrac12 . \qedhere
\]
\end{proof}

\subsection{Proof of Theorem~\ref{thm:A}}\label{sec:proofA}

\begin{proof}[Proof of Theorem~\ref{thm:A}]
Let $\alpha\in(\frac12,1)$ and let us fix $p\in[1,2\alpha)$, a non-empty range by
Proposition~\ref{prop:compactA}; observe that $p<2\alpha<2$. Hypothesis (a) of
Proposition~\ref{prop:package} is \eqref{eq:M2A} and hypothesis (b) is
Proposition~\ref{prop:compactA}(i) for that $p$. From Proposition~\ref{prop:package} we therefore
obtain a subsequence converging strongly in $L^{p}\bigl((0,T);\LL\bigr)$ and weakly in $\LHone$,
and a function $u\in\WscaleG$ which is a weak solution of \eqref{eq:P};
Theorem~\ref{thm:uniqueness} makes it the only one. This is (i).

For (ii) no further extraction is needed, and the conclusion is about the whole sequence. Let
$p\in[1,2\alpha)$. Every subsequence of $\{u_{n}\}$ satisfies \eqref{eq:M2A} and
Proposition~\ref{prop:compactA}(i) verbatim, so Proposition~\ref{prop:package} applies to it and
yields a further subsequence converging strongly in $L^{p}\bigl((0,T);\LL\bigr)$ to a weak
solution, which is $u$ by Theorem~\ref{thm:uniqueness}. Every subsequence therefore has a further
subsequence converging to $u$ in that norm, and hence the whole sequence does; the exponent $p$
was arbitrary, and the same argument gives the weak convergence in $\LHone$. The bound
\eqref{eq:spectralbound} with $q=2$, admissible because $\alpha>\frac12$, is uniform in $n$, so
the same subsequence argument applied in $\LinfL$, the dual of $L^{1}\bigl((0,T);\LL\bigr)$,
yields the weak-$*$ convergence of the whole sequence as well.

Part (iii) is Proposition~\ref{prop:spectralbound} with $q=2$, which is admissible because
$2>1/\alpha$ for $\alpha>\frac12$: part (i) of that proposition bounds the approximations
uniformly, and part (ii), whose hypothesis is the weak convergence just established, carries the
bound to $u$.
\end{proof}

\begin{remark}[The window and the convergence are two different things]\label{rem:window}
The interval $[1,2\alpha)$ is the range of exponents for which the compactness criterion can be
\emph{applied} with the bound \eqref{eq:M2A}; it is not the range in which convergence is
\emph{obtained}. Indeed $\{u_{n}\}$ is bounded in $\LLtwo$ by \eqref{eq:M2A} and Poincar\'e, so
for $p\in[1,2\alpha)$ and $s\in(p,2)$, with $\theta\in(0,1)$ defined by
$s^{-1}=\theta p^{-1}+(1-\theta)2^{-1}$,
\[
\norm{u_{n}-u}_{L^{s}((0,T);\LL)}
\le\norm{u_{n}-u}^{\theta}_{L^{p}((0,T);\LL)}\,
\norm{u_{n}-u}^{1-\theta}_{\LLtwo},
\]
whose second factor is bounded; the convergence of Theorem~\ref{thm:A}(ii) therefore holds in
$L^{s}\bigl((0,T);\LL\bigr)$ for every $s<2$. We state the theorem with $[1,2\alpha)$ because
that interval, and not $[1,2)$, is what the criterion sees, and it is the interval whose
non-emptiness produces the restriction on $\alpha$.
\end{remark}

\subsection{Consequences}\label{sec:consequencesA}

The solution inherits the bounds of Proposition~\ref{prop:boundsA}. Norms are weakly lower
semicontinuous, so \eqref{eq:N1} and \eqref{eq:M2A} pass to the limit and we obtain
\begin{equation}\label{eq:limitbounds}
\norm{u}^{2}_{\LLtwo}\le N_{1},\qquad\norm{u}^{2}_{\LHone}\le M_{2} .
\end{equation}
The strong convergence is itself information about the approximating sequence that a purely weak
argument would not supply, and it is what the compactness route is for.

The restriction $\alpha>\frac12$ is one of the method and not of the problem.
Proposition~\ref{prop:compactA}
states that the compactness criterion, applied with the only bound available here, admits no
exponent when $\alpha\le\frac12$. It does not state, and cannot state, that no solution exists.
In fact one does: Zacher \cite{zacher09} proves existence and uniqueness in this scale for every
$\alpha\in(0,1)$, under the weaker hypothesis $f\in L^{2}\bigl((0,T);\HM\bigr)$ and with merely
bounded measurable coefficients.

That argument uses no compactness at all, and it need not, since \eqref{eq:P} is linear. The
three identities we traverse in the passage to the limit, \eqref{eq:ibp-n},
\eqref{eq:galerkinweak} and \eqref{eq:limitequation}, are linear in $u_{n}$. In each of them we
pair the unknown with a fixed object and never with itself inside a nonlinearity, and weak
convergence suffices to pass to the limit in a linear pairing against a continuous functional.
One needs strong compactness only when the equation carries a nonlinear term, such as a product
$u_{n}\nabla u_{n}$ or a power $\abs{u_{n}}^{p-2}u_{n}$, because composition with a nonlinear
function does not preserve weak convergence.

The restriction can be traced. The route gives us only
the differential inequality \eqref{eq:diffineq}; any attempt to extract a pointwise value in $t$
from it ends in evaluating $\IL{\alpha}\norm{f}^{2}_{\LL}$ at a point, and
Proposition~\ref{prop:nobound} forbids that. Our energy bounds are therefore averaged, the
compactness criterion sees only the energy norm, and the window of exponents closes at
$\alpha=\frac12$. The
cut is thus a phenomenon of small orders and worsens as $\alpha$ decreases, for the same reason
the concentration argument works: the more singular the kernel, the more mass an $L^{2}$ source
can pile up where the weight is largest, and the less the energy scale can absorb.

Whether $\alpha=\frac12$ means anything for the problem itself, rather than for this
construction, is a separate question with a separate answer, and we settle it in
Section~\ref{sec:sharp}.

\section{The case \texorpdfstring{\Hthreep}{(H3')}: the stronger hypothesis}\label{sec:caseB}

\subsection{Statements}

\begin{theorem}[Existence and regularity under \Hthreep]\label{thm:B}
Let $\alpha\in(0,1)$ and assume \Hone, \Htwo, \Hthreep{} and \Hfour. Then
\begin{enumerate}[label=\textup{(\roman*)},leftmargin=2.4em]
\item problem \eqref{eq:P} has a weak solution $u\in\WscaleG$;
\item that solution satisfies
\begin{equation}\label{eq:extrareg}
u\in\LinfL,
\end{equation}
with the explicit bound \eqref{eq:M1} below;
\item the Galerkin approximations converge to $u$, the whole sequence, strongly in $\LLtwo$,
weakly in $\LHone$ and weakly-$*$ in $\LinfL$.
\end{enumerate}
\end{theorem}

The solution is unique by Theorem~\ref{thm:uniqueness}. Conclusion (ii) also follows from
Proposition~\ref{prop:spectralbound} with $q=r$, in a simpler form: no Mittag-Leffler factor, no
free parameter, and the bound in the norm rather than in its square. We keep \eqref{eq:M1}
because it is what the energy route produces and what Theorem~\ref{thm:stability} uses. What
\Hthreep{} contributes to (ii) is not the conclusion, which Theorem~\ref{thm:A}(iii) already
gives for $\alpha>\frac12$, but its validity in the range $\alpha\le\frac12$, where by
Proposition~\ref{prop:sharp} it is false under \Hthree.

\begin{corollary}[The uniform bound at the exponent
\texorpdfstring{$1/\alpha$}{1/alpha}]\label{cor:oneoveralpha}
Let $\alpha\in(0,1)$, assume \Hone, \Htwo{} and \Hfour, and suppose the hypotheses of
Theorem~\ref{thm:A} or those of Theorem~\ref{thm:B} to hold, so that \eqref{eq:P} has a weak
solution $u$ obtained as a limit of the Galerkin approximations. If in addition $f\in\LLq$ for
some $q>1/\alpha$, then $u\in\LinfL$ and
\begin{equation}\label{eq:Linfq}
\norm{u}_{\LinfL}\le\norm{g}_{\LL}+\kappa_{\alpha}(q,T)\norm{f}_{\LLq},
\end{equation}
with $\kappa_{\alpha}(q,T)$ the constant of \eqref{eq:mapping}.
\end{corollary}

\begin{proof}
Theorem~\ref{thm:A}(ii) and Theorem~\ref{thm:B}(iii) each supply the weak convergence in $\LHone$
that Proposition~\ref{prop:spectralbound}(ii) asks for, and the conclusion is that proposition
applied with the exponent $q$.
\end{proof}

Two numbers that the statement of Theorem~\ref{thm:B} keeps together are thereby separated. The
conclusion $u\in\LinfL$ requires only $q>1/\alpha$, and requires it at every order; the exponent
$r>2/\alpha$ of \Hthreep{} is what the energy route of Section~\ref{sec:boundsB} consumes in
reaching that same conclusion with the explicit constant \eqref{eq:M1}, which
Theorem~\ref{thm:stability} then converts into a stability estimate. Neither number, however, is
the optimal integrability for the existence of a solution, and the corollary should not be read as
if it were: it presupposes a solution and improves its regularity. For $q$ in the interval
$\bigl(1/\alpha,2/\alpha\bigr]$ we offer no construction, the passage to the limit of
Sections~\ref{sec:windowA} and~\ref{sec:caseB} resting on the energy inequalities and on nothing
weaker. Whether that gap can be closed is among the questions left open in
Section~\ref{sec:conclusions}.

\begin{theorem}[Stability]\label{thm:stability}
Let $\alpha\in(0,1)$ and $\varepsilon>0$, assume \Hone, \Htwo, \Hthreep{} and \Hfour, and
let $U_{1},U_{2}$ be the weak solutions associated with two pairs of data. Write $W:=U_{1}-U_{2}$,
$\tilde g:=g^{(1)}-g^{(2)}$ and $\tilde f:=f^{(1)}-f^{(2)}$. Then
\begin{equation}\label{eq:stability}
\norm{W}_{\LinfL}\le K_{\varepsilon}(T,\alpha)
\Bigl(\norm{\tilde g}_{\LL}+\norm{\tilde f}_{\LLr}\Bigr),
\end{equation}
where
\begin{equation}\label{eq:Keps}
K_{\varepsilon}(T,\alpha):=\left[E_{\alpha}\bigl(\varepsilon T^{\alpha}\bigr)
\max\Bigl\{1,\tfrac{1}{\varepsilon}\kappa_{\alpha}\!\left(\tfrac r2,T\right)\Bigr\}
\right]^{1/2},
\end{equation}
and $\kappa_{\alpha}(r/2,T)$ is the constant of \eqref{eq:kappa}.
\end{theorem}

The parameter is free, so \eqref{eq:stability} holds with the infimum of
$K_{\varepsilon}(T,\alpha)$ over $\varepsilon>0$; the two entries of the maximum balance at
$\varepsilon=\kappa_{\alpha}(r/2,T)$, which need not be the optimal choice, the Mittag-Leffler
factor being increasing in $\varepsilon$.

\begin{corollary}[Well-posedness in the sense of Hadamard]\label{cor:hadamard}
Let $\alpha\in(0,1)$ and assume \Hone, \Htwo, \Hthreep{} and \Hfour. Then problem \eqref{eq:P}
is well posed in the sense of Hadamard: a weak solution exists, it is unique, and it depends
continuously on the data in the norms of \eqref{eq:stability}.
\end{corollary}

Since $\Hthreep\Rightarrow\Hthree$, the whole of Section~\ref{sec:common} applies without change,
and so does Proposition~\ref{prop:boundsA}, whose proof used nothing beyond \Hthree. What remains is what the
reinforcement adds: a bound uniform in $t$, now for every order.

\subsection{The uniform bound}\label{sec:boundsB}

\begin{proposition}[Uniform bounds under \Hthreep]\label{prop:boundsB}
Let $\alpha\in(0,1)$ and assume \Hone, \Htwo, \Hthreep{} and \Hfour. Then
\begin{enumerate}[label=\textup{(\roman*)},leftmargin=2.4em]
\item for every $\varepsilon>0$ and every $n$,
\begin{equation}\label{eq:M1}
\sup_{t\in[0,T]}\norm{u_{n}(t)}^{2}_{\LL}\le M_{1}(\varepsilon)
:=E_{\alpha}\bigl(\varepsilon T^{\alpha}\bigr)
\left(\norm{g}^{2}_{\LL}
+\frac{1}{\varepsilon}\kappa_{\alpha}\!\left(\tfrac r2,T\right)\norm{f}^{2}_{\LLr}\right);
\end{equation}
\item there is $M_{2}'>0$, independent of $n$, with
$\int_{0}^{T}\norm{u_{n}(t)}^{2}_{\HH}dt\le M_{2}'$.
\end{enumerate}
\end{proposition}

\begin{proof}
\emph{Step 1: the inequality is already of Gr\"onwall type.} Under \Hthreep{},
Proposition~\ref{prop:threshold} bounds the source term of \eqref{eq:energyL2} uniformly in $t$,
and with $\norm{g_{n}}_{\LL}\le\norm{g}_{\LL}$ from \eqref{eq:gnconv} that inequality becomes
\[
\norm{u_{n}(t)}^{2}_{\LL}\le
\underbrace{\norm{g}^{2}_{\LL}
+\frac{1}{\varepsilon}\kappa_{\alpha}\!\left(\tfrac r2,T\right)\norm{f}^{2}_{\LLr}}_{=:a}
+\varepsilon\bigl(\IL{\alpha}\norm{u_{n}}^{2}_{\LL}\bigr)(t),\qquad t\in[0,T],
\]
which is \eqref{eq:gronwall-hyp} with $\gamma=\alpha$, coefficient $\varepsilon$ and independent
term the constant $a$, which is in particular non-decreasing, as Lemma~\ref{lem:gronwall}
requires. No integration in time is needed, and that is the difference with the previous
case.

\emph{Step 2: Gr\"onwall.} Lemma~\ref{lem:gronwall} with $y=\norm{u_{n}(\cdot)}^{2}_{\LL}$ gives
$\norm{u_{n}(t)}^{2}_{\LL}\le a\,E_{\alpha}(\varepsilon t^{\alpha})$; since $E_{\alpha}$ has
positive coefficients it increases on $[0,\infty)$, so that the supremum over $[0,T]$ is bounded
by $a\,E_{\alpha}(\varepsilon T^{\alpha})$, which is \eqref{eq:M1}.

\emph{Step 3: the bound in $\HH$.} We take $\varepsilon=1$ in \eqref{eq:energyH1} and evaluate at
$t=T$, which is licit because in the regime $f\in\mathcal{D}$ of Remark~\ref{rem:density} both
sides are continuous on $[0,T]$. On the right, the first summand is at most $\norm{g}^{2}_{\LL}$,
the second at most $\kappa_{\alpha}(r/2,T)\norm{f}^{2}_{\LLr}$ by
Proposition~\ref{prop:threshold}, and the third we treat by \eqref{eq:M1} together with
\[
\bigl(\IL{\alpha}\norm{u_{n}}^{2}_{\LL}\bigr)(T)
\le\frac{\sup_{s}\norm{u_{n}(s)}^{2}_{\LL}}{\Gamma(\alpha)}\int_{0}^{T}(T-s)^{\alpha-1}ds
=\frac{T^{\alpha}}{\Gamma(\alpha+1)}\sup_{s}\norm{u_{n}(s)}^{2}_{\LL}.
\]
On the left, since $\alpha<1$ the function $x\mapsto x^{\alpha-1}$ decreases, so
$(T-s)^{\alpha-1}\ge T^{\alpha-1}$ for $s\in(0,T)$ and therefore
\[
\bigl(\IL{\alpha}\norm{u_{n}}^{2}_{\HH}\bigr)(T)\ge\frac{T^{\alpha-1}}{\Gamma(\alpha)}
\int_{0}^{T}\norm{u_{n}(s)}^{2}_{\HH}ds .
\]
Collecting the two sides, we obtain
\[
\frac{2a_{0}T^{\alpha-1}}{\Gamma(\alpha)}\int_{0}^{T}\norm{u_{n}(s)}^{2}_{\HH}ds
\le\norm{g}^{2}_{\LL}+\kappa_{\alpha}\!\left(\tfrac r2,T\right)\norm{f}^{2}_{\LLr}
+\frac{T^{\alpha}}{\Gamma(\alpha+1)}M_{1}(1),
\]
and solving for the integral we obtain the second bound, the factor on the left being a positive
constant independent of $n$.
\end{proof}

\subsection{Compactness with the exponent \texorpdfstring{$p=2$}{p=2}}

\begin{proposition}[Verification of \eqref{eq:LL1} under \Hthreep]\label{prop:compactB}
Let $\alpha\in(0,1)$, assume \Hthreep, and take the choices \eqref{eq:triple}. Then
\eqref{eq:LL1} holds with $p=2$.
\end{proposition}

\begin{proof}
With the uniform bound \eqref{eq:M1} available we no longer need the detour of the previous
section. We take $\varepsilon=1$ in \eqref{eq:energyH1} and bound its right-hand side as in
Step 3 above:
\[
\bigl(\IL{\alpha}\norm{u_{n}}^{2}_{\HH}\bigr)(t)\le\frac{1}{2a_{0}}
\left(\norm{g}^{2}_{\LL}+\kappa_{\alpha}\!\left(\tfrac r2,T\right)\norm{f}^{2}_{\LLr}
+\frac{T^{\alpha}}{\Gamma(\alpha+1)}M_{1}(1)\right),
\]
uniformly in $t$ and in $n$, which is \eqref{eq:LL1} with $p=2$. No condition on $\alpha$ appears
anywhere in this computation.
\end{proof}

\subsection{Proofs}

\begin{proof}[Proof of Theorem~\ref{thm:B}]
Let $\alpha\in(0,1)$. Hypothesis (a) of Proposition~\ref{prop:package} is the second bound of
Proposition~\ref{prop:boundsB}, hypothesis (b) with $p=2$ is Proposition~\ref{prop:compactB}, and
the hypothesis \Hthree{} that Proposition~\ref{prop:package} requires holds by
Proposition~\ref{prop:hierarchy}(i). We therefore obtain a subsequence and a weak solution
$u\in\WscaleG$, unique by Theorem~\ref{thm:uniqueness}. This is (i).

For the additional regularity, \eqref{eq:M1} is uniform in $n$ and $\LinfL$ is the dual of
$L^{1}\bigl((0,T);\LL\bigr)$, so that bounded sets there are weak-$*$ sequentially compact: a
further subsequence converges weakly-$*$ in $\LinfL$, and its limit is $u$ by uniqueness of
limits in $\mathcal{D}'\bigl((0,T);\LL\bigr)$. Lower semicontinuity of the norm with respect to
weak-$*$ convergence then gives
\[
\norm{u}^{2}_{\LinfL}\le\liminf_{n}\ \sup_{t\in[0,T]}\norm{u_{n}(t)}^{2}_{\LL}
\le M_{1}(\varepsilon)
\]
for every $\varepsilon>0$, which is (ii) with the bound \eqref{eq:M1}.

For (iii), the three convergences hold for the whole sequence by the argument used in
Theorem~\ref{thm:A}: every subsequence of $\{u_{n}\}$ satisfies the same uniform bounds, so
Proposition~\ref{prop:package} with $p=2$, and the weak-$*$ extraction just performed, produce
out of it a further subsequence converging to a weak solution, which is $u$ by
Theorem~\ref{thm:uniqueness}.
\end{proof}

\begin{proof}[Proof of Theorem~\ref{thm:stability}]
By Corollary~\ref{cor:linear}, $W$ is the weak solution for the data $(\tilde g,\tilde f)$ and
coincides with the function constructed above for those data. Bound \eqref{eq:M1} holds uniformly
in $n$ and passes to the limit exactly as in the proof of Theorem~\ref{thm:B}, so
\[
\norm{W}^{2}_{\LinfL}\le E_{\alpha}\bigl(\varepsilon T^{\alpha}\bigr)
\Bigl(\norm{\tilde g}^{2}_{\LL}
+\tfrac{1}{\varepsilon}\kappa_{\alpha}\!\left(\tfrac r2,T\right)
\norm{\tilde f}^{2}_{\LLr}\Bigr).
\]
Bounding the bracket by its largest coefficient and using
$\sqrt{x^{2}+y^{2}}\le x+y$ for $x,y\ge0$, we obtain \eqref{eq:stability}.
\end{proof}

\begin{proof}[Proof of Corollary~\ref{cor:hadamard}]
Existence is Theorem~\ref{thm:B}, uniqueness Theorem~\ref{thm:uniqueness}, continuous dependence
Theorem~\ref{thm:stability}.
\end{proof}

\section{Comparison of the two cases}\label{sec:comparison}

With both developments complete we can set the two cases side by side. We start the comparison
where Section~\ref{sec:relation} left it,
\[
\boxed{\ \Hthreep\Longrightarrow\Hthree\ },
\]
so that \Hthree{} is the weaker hypothesis and \Hthreep{} the stronger: every source admitted by
the second is admitted by the first, and by Proposition~\ref{prop:hierarchy}(ii) not conversely.

\medskip
\noindent
\begin{tabular}{@{}>{\raggedright\arraybackslash}p{0.235\textwidth}>{\raggedright\arraybackslash}p{0.35\textwidth}>{\raggedright\arraybackslash}p{0.35\textwidth}@{}}
\toprule
 & \Hthree{} (weaker) & \Hthreep{} (stronger)\\
\midrule
Condition on $f$ & $f\in\LLtwo$ & $f\in\LLr$, $r>2/\alpha$; implies the condition on the left\\[4pt]
Control of $\bigl(\IL{\alpha}\norm{f}^{2}_{\LL}\bigr)(t)$ & none: no such constant exists
(Proposition~\ref{prop:nobound}) & $\le\kappa_{\alpha}(r/2,T)\norm{f}^{2}_{\LLr}$, uniformly in
$t$ (Proposition~\ref{prop:threshold})\\[4pt]
Gr\"onwall applied to & the primitive $Y_{n}$, after integrating in $t$ &
$\norm{u_{n}(t)}^{2}_{\LL}$ itself\\[4pt]
Uniform bounds obtained & \eqref{eq:N1} and \eqref{eq:M2A}, averaged in $t$ & \eqref{eq:M1},
uniform in $t$, and the $\HH$ bound of Proposition~\ref{prop:boundsB}\\[4pt]
Compactness exponent & $p\in[1,2\alpha)$, non-empty iff $\alpha>1/2$ & $p=2$, for every
$\alpha\in(0,1)$\\[4pt]
Convergence of the approximations & strongly in $L^{p}\bigl((0,T);\LL\bigr)$ for every $p<2$
(Remark~\ref{rem:window}); weakly in $\LHone$ & strongly in $\LLtwo$; weakly in $\LHone$;
weakly-$*$ in $\LinfL$\\[4pt]
Range of $\alpha$ & $(1/2,\,1)$ & all of $(0,1)$\\[4pt]
Regularity obtained & $\WscaleG\cap\LinfL$; the second summand by the mode representation
(Theorem~\ref{thm:A}(iii)), not by the energy route & $\WscaleG\cap\LinfL$, with the explicit
constant \eqref{eq:M1}\\[4pt]
Stability & not reached along this route & \eqref{eq:stability}, constant written out\\
\bottomrule
\end{tabular}

\medskip
The hypotheses are ordered; the theorems are not. Theorem~\ref{thm:A} assumes strictly less of
the source, and what it loses is the segment $\alpha\in(0,\frac12]$. It does not lose the bound
in $\LinfL$: the energy route of Section~\ref{sec:caseA} yields none, but the mode representation
does, by Theorem~\ref{thm:A}(iii). Neither statement
contains the other, because neither set of admissible pairs $(\alpha,f)$ contains the other: the
first covers a source in $\LLtwo$ with no better integrability in time, which the second does
not, and the second covers an order $\alpha\le\frac12$, which the first does not. Where both
apply the second adds the explicit stability constant and nothing else.

What differs between the two columns is the norm in which we bound the approximating sequence,
and the difference propagates to the limit. Under the weaker
hypothesis the compactness criterion can be applied only with an exponent $p<2\alpha\le2$, so
that what it delivers directly is strong convergence in $L^{p}\bigl((0,T);\LL\bigr)$, in the
worst case merely in $L^{1}$; this already suffices for Lemma~\ref{lem:identification}, since
$\RLD\psi$ belongs to every $L^{q}\bigl((0,T);\HH\bigr)$, and interpolation against the energy
bound then upgrades it to every exponent below $2$, by Remark~\ref{rem:window}. Under the
stronger hypothesis the criterion applies with the exponent $2$ outright and the weak-$*$ bound
is preserved.

\subsection*{Classification of the thresholds}

A condition is a \emph{restriction of the method} if the conclusion holds without it. The cut $\alpha>\frac12$ of Theorem~\ref{thm:A} is of this
kind. Proposition~\ref{prop:compactA} shows that the compactness criterion, fed with the only
bound the weaker hypothesis yields, admits no exponent for $\alpha\le\frac12$; but existence
holds for every $\alpha\in(0,1)$ under \Hthree, indeed under less \cite{zacher09}, by an argument
that needs no compactness because the equation is linear. What we obtain in exchange from the
compactness route is the strong convergence \eqref{eq:strongconv}, and $\alpha>\frac12$ is its
price. The same label fits the exponent $2/\alpha$ of \Hthreep: Proposition~\ref{prop:threshold}
shows it to be exactly what the energy argument consumes, and Proposition~\ref{prop:spectralbound} reaches the same conclusion from the weaker requirement
$q>1/\alpha$ by a different route.

A condition marks a \emph{frontier of the problem} when the conclusion is false without it. The
threshold $\alpha=\frac12$ in the regularity \eqref{eq:extrareg} under \Hthree{} is of this kind,
and we prove it in Section~\ref{sec:sharp}: for every $\alpha\in(0,\frac12]$ there is a source
satisfying \Hthree{} whose unique weak solution leaves $\LinfL$. No refinement of any argument
can recover the conclusion there, because there is nothing to recover.

The value $\alpha=\frac12$ occurs at three places in this paper, with a different meaning at
each. In Remark~\ref{rem:pairunique} it separates the Riemann--Liouville from the Caputo reading
of $\partial^{\alpha}_{t}$; it concerns the representation of the derivative, and it is why the
scale is defined as it is. In Theorem~\ref{thm:A} it is where the window $[1,2\alpha)$ of
Proposition~\ref{prop:compactA} closes, and it comes from the construction. In
Proposition~\ref{prop:sharp} it is where the regularity \eqref{eq:extrareg} ceases to hold, and
it is a property of the equation. That the three numbers agree is a coincidence of the exponents
involved. The loss of the segment $\alpha\in(0,\frac12]$ in Theorem~\ref{thm:A} is
therefore not a property of the problem, and the failure of \eqref{eq:extrareg} under \Hthree{}
for $\alpha\le\frac12$ is not a technical gap awaiting a sharper estimate.

\section{Optimality of the threshold \texorpdfstring{$\alpha=\frac12$}{alpha=1/2}}
\label{sec:sharp}

The energy method cannot settle whether the regularity \eqref{eq:extrareg}, which the energy
route does not deliver under \Hthree, fails because of the route or because it is false. The mode
representation of Section~\ref{sec:spectral} settles it: $\alpha=\frac12$ is the exact frontier,
the regularity holding above that value and failing below it and at it.

\begin{proposition}[Optimality of $\alpha=\frac12$]\label{prop:sharp}
Assume \Hone{} and \Htwo, and let $g=0$. Then
\begin{enumerate}[label=\textup{(\roman*)},leftmargin=2.4em]
\item if $\alpha\in\bigl(\frac12,1\bigr)$ and $f\in\LLtwo$, then the weak solution satisfies
$u\in\LinfL$;
\item if $\alpha\in\bigl(0,\frac12\bigr]$, then there is $f\in\LLtwo$, of the form
$f(x,t)=F(t)w_{1}(x)$, whose weak solution does not belong to $\LinfL$.
\end{enumerate}
\end{proposition}

\begin{proof}
(i) is Theorem~\ref{thm:A}(iii) with $g=0$, restated here for the comparison with (ii); the
exponent $2$ is admissible in Proposition~\ref{prop:spectralbound} exactly because
$\alpha>\frac12$.

(ii) We write $k:=k_{1}$ and $\theta:=\theta_{1}>0$. If $\alpha\le\frac12$ then
$2\alpha-2\le-1$, so the lower bound in \eqref{eq:kernelbounds} gives
\[
\int_{0}^{T}k(\sigma)^{2}d\sigma\ge\theta^{2}\int_{0}^{T}\sigma^{2\alpha-2}d\sigma=\infty,
\]
that is $k\notin L^{2}(0,T)$. Suppose now, for contradiction, that $k*F\in L^{\infty}(0,T)$ for
every $F\in L^{2}(0,T)$. The operator $SF:=k*F$ is then defined on all of $L^{2}(0,T)$, linear from $L^{2}(0,T)$ to
$L^{\infty}(0,T)$, and closed: if $F_{m}\to F$ in $L^{2}$ and $SF_{m}\to G$ in $L^{\infty}$, then
Young's inequality gives $SF_{m}\to SF$ in $L^{1}$, so $G=SF$ a.e. By the closed graph theorem
there is $M$ with $\norm{k*F}_{L^{\infty}}\le M\norm{F}_{L^{2}}$ for every $F$. We test this on
$F_{m}(s):=\min\{k(T-s),m\}$, which lies in $L^{2}\cap L^{\infty}$. Since $k>0$, the change
$\sigma=T-s$ gives
\[
(k*F_{m})(T)=\int_{0}^{T}k(\sigma)\min\{k(\sigma),m\}d\sigma
\ge\int_{0}^{T}\min\{k(\sigma),m\}^{2}d\sigma=\norm{F_{m}}^{2}_{L^{2}},
\]
and $k*F_{m}$ is continuous, being the convolution of an $L^{1}$ function with an $L^{\infty}$
one, so $\norm{k*F_{m}}_{L^{\infty}}\ge\norm{F_{m}}^{2}_{L^{2}}$ and
\[
M\ge\frac{\norm{k*F_{m}}_{L^{\infty}}}{\norm{F_{m}}_{L^{2}}}\ge\norm{F_{m}}_{L^{2}}
\xrightarrow[m\to\infty]{}\norm{k}_{L^{2}(0,T)}=\infty
\]
by monotone convergence, a contradiction. There is therefore $F\in L^{2}(0,T)$ with
$k*F\notin L^{\infty}(0,T)$.

We set $f(x,t):=F(t)w_{1}(x)$, which satisfies \Hthree, and $g=0$, which satisfies \Hfour. We
claim that $u(t):=(k*F)(t)w_{1}$ is the weak solution for these data. That $u\in\LHone$ follows
from Young's inequality, $k$ being in $L^{1}(0,T)$ and $F$ in $L^{2}(0,T)$. To identify it, we
approximate: we choose $F_{m}\in C^{\infty}_{0}(0,T)$ with $F_{m}\to F$ in $L^{2}(0,T)$ and set
$f_{m}:=F_{m}w_{1}\in\mathcal{D}$. For this source every mode with $j\ge2$ vanishes, so by
Lemma~\ref{lem:modesofu} the Galerkin approximations are $u^{(m)}=(k*F_{m})w_{1}$ for every
$n\ge1$; by Lemma~\ref{lem:modes}(ii) this function is absolutely continuous and satisfies
\eqref{eq:galerkin} pointwise with
$\partial^{\alpha}_{t}u^{(m)}=(F_{m}-\lambda_{1}k*F_{m})w_{1}$, so the integration by parts
\eqref{eq:intbyparts} turns \eqref{eq:galerkin} into the weak formulation \eqref{eq:weak} with
data $(0,f_{m})$. Letting $m\to\infty$, Young's inequality gives $k*F_{m}\to k*F$ in
$L^{2}(0,T)$, hence $u^{(m)}\to u$ in $\LHone$ and
$\partial^{\alpha}_{t}u^{(m)}\to(F-\lambda_{1}k*F)w_{1}$ in $\LHmone$, while $f_{m}\to f$ in
$\LLtwo$; every term of \eqref{eq:weak} is linear and continuous under these convergences, so
that the identity passes to the limit. Hence $u\in\WscaleG$ is a weak solution and, by
Theorem~\ref{thm:uniqueness}, the weak solution. Since $\norm{w_{1}}_{\LL}=1$ we have
$\norm{u(t)}_{\LL}=\abs{(k*F)(t)}$, which is not in $L^{\infty}(0,T)$.
\end{proof}

\begin{remark}[Comparison with the single-phase-lag equation]\label{rem:spl}
Proposition~\ref{prop:sharp} may appear to be in tension with results available for neighbouring
equations. For the fractional single-phase-lag model, Maes and
Van Bockstal \cite{maes} obtain a weak solution in
$L^{\infty}\bigl((0,T);\HH\bigr)\cap C\bigl([0,T];\LL\bigr)$ under $F\in\LLtwo$, which is our
\Hthree{} verbatim, and for every $\alpha\in(0,1)$, with no threshold on the order. Their
equation is not \eqref{eq:P}, however: besides the fractional derivative it carries the classical
term $\rho c\,\partial_{t}u$ with $\rho c>0$, and their estimate is obtained by testing with
$\partial_{t}u$, which that term makes possible. The uniform bound is supplied there by the
first-order derivative and not by the fractional one.

For \eqref{eq:P} no such test function is available: $\partial_{t}u$ need not belong to
$\LLtwo$, let alone to $\HH$, and every route to a pointwise value in $t$ must pass through
$\IL{\alpha}$, whose mapping property into $L^{\infty}$ fails at the exponent $2$ for every
$\alpha<\frac12$, by Proposition~\ref{prop:nobound}, and fails at $\alpha=\frac12$ as well,
which Proposition~\ref{prop:sharp}(ii) establishes by a different argument since the
concentration family does not reach that endpoint. The bound is missing not because our argument
is weak, but because the equation has no term able to produce it.
\end{remark}

\subsection*{Comparison with other approaches}

Two other approaches meet the same exponent. Floridia, G\"olgeleyen and Yamamoto \cite{floridia}
construct strong solutions of $\partial^{\alpha}_{t}(u-a)=Au+F$ in a Banach space $X$ under
$F\in L^{q}(0,T;X)$. The condition their proof consumes is written there as a restriction on the
fractional power in which the initial datum is taken, and for a datum in $X$ itself it reduces to
$q>1/\alpha$. The mechanism is ours: the boundedness of
$\IL{\alpha}:L^{p}(0,T)\to L^{\infty}(0,T)$ for $p>1/\alpha$. The factor two in $r>2/\alpha$
reflects only that our estimates are carried out at the level of the squared norm. Their framework
requires $A$ to be sectorial, which for the elliptic realisation forces smooth coefficients and a
smooth boundary, so it does not contain the present setting; but the threshold is the same object
seen twice.

The second line lies further from ours. Kubica, Ryszewska and Zacher \cite{kubicaRZ}
work with De Giorgi--Nash--Moser techniques and prove a weak Harnack inequality and interior
H\"older continuity for $f\in L^{q_{1}}\bigl((0,T);L^{q_{2}}(\Omega)\bigr)$ under the balance
\cite[(H3)]{kubicaRZ}, which we quote in their notation, where $N$ is the dimension, written $d$
above, and $d$ is a slack parameter: for $N\ge2$,
\begin{equation}\label{eq:KRZ}
\frac{p_{0}'}{q_{1}}+\frac{N}{2q_{2}}=1-d,
\qquad
q_{1}\ge\frac{p_{0}'}{1-d},
\qquad
q_{2}\ge\frac{N}{2(1-d)},
\qquad
d\in(0,1),
\end{equation}
with $p_{0}$ the exponent in which the kernel $l$ conjugate to $k$ is integrable. In the
single-order case \cite[Eq.~(3)]{kubicaRZ}, where $k(t)=t^{-\alpha}/\Gamma(1-\alpha)$ and
$l(t)=t^{\alpha-1}/\Gamma(\alpha)$, any $p_{0}<1/(1-\alpha)$ is admissible
\cite[Rem.~1.2]{kubicaRZ}, so $p_{0}'$ may be brought arbitrarily close to $1/\alpha$ from above,
but not to it. Take $q_{2}=2$, the space in which we measure the source. For $N=2$ the spatial
term of \eqref{eq:KRZ} equals $\frac12$, so that $p_{0}'/q_{1}=\frac12-d$; this forces
$d\in(0,\frac12)$, the two side conditions then holding automatically, and
\[
q_{1}=\frac{p_{0}'}{\tfrac12-d}\in\bigl(2p_{0}',\infty\bigr),
\qquad\text{whose infimum is}\qquad
\inf q_{1}=\frac{2}{\alpha} .
\]
Hypothesis \Hthree, which is $q_{1}=2$, lies outside that range. An energy estimate and a Moser
iteration share no technique, and the exponent at which each becomes available is the same number.
The agreement is dimensional, and that is itself informative: in \eqref{eq:KRZ} the choice
$q_{2}=2$ requires $1-d-N/4>0$ and is therefore available only for $N\le3$, the infimum being
$4/\alpha$ when $N=3$. Our $2/\alpha$ carries no dimension, the source being measured in $\LL$ and
no Sobolev embedding in space being used anywhere; theirs carries $N$ because the Moser iteration
uses one, and $N=2$ is where that weight is neutral. Proposition~\ref{prop:spectralbound} shows that for the operator considered here the conclusion
needs only $q>1/\alpha$; we retain \Hthreep{} because
it is the hypothesis that survives when the coefficient depends on time, in the setting of
\cite{kubica}, where the spectral representation is unavailable, and because it produces the
explicit constant of Theorem~\ref{thm:stability}.

\section{Examples}\label{sec:examples}

One example is called for: a function of the solution space whose norm we compute, which shows
that the two summands of that norm measure different things.

\begin{example}[A function of $W^{\alpha}_{2,2}$ and the value of its norm]\label{ex:space}
Let $w_{1}$ be the first eigenfunction of Lemma~\ref{lem:spectral}, normalised by
$\norm{w_{1}}_{\LL}=1$, and let
\[
u(x,t):=\frac{t^{\alpha}}{\Gamma(\alpha+1)}\,w_{1}(x),\qquad (x,t)\in\Omega\times[0,T].
\]
For the first summand of the norm, $u(\cdot,t)\in\HH$ for each $t$ with
$\norm{u(\cdot,t)}_{\HH}=t^{\alpha}\Gamma(\alpha+1)^{-1}\norm{w_{1}}_{\HH}$, so
\[
\norm{u}^{2}_{\LHone}=\frac{\norm{w_{1}}^{2}_{\HH}}{\Gamma(\alpha+1)^{2}}\int_{0}^{T}t^{2\alpha}dt
=\frac{T^{2\alpha+1}}{(2\alpha+1)\Gamma(\alpha+1)^{2}}\norm{w_{1}}^{2}_{\HH}<\infty .
\]
For the second, we write $\varphi(t):=t^{\alpha}/\Gamma(\alpha+1)$. Then
$\varphi'(t)=\alpha t^{\alpha-1}/\Gamma(\alpha+1)$ is integrable because $\alpha-1>-1$, so
$\varphi\in AC([0,T])$ and \eqref{eq:caputo} applies:
\[
\partial^{\alpha}_{t}\varphi=\bigl(\IL{1-\alpha}\varphi'\bigr)(t)
=\frac{\alpha}{\Gamma(1-\alpha)\Gamma(\alpha+1)}\int_{0}^{t}(t-s)^{-\alpha}s^{\alpha-1}ds .
\]
The integral is Euler's integral of the first kind,
$\int_{0}^{t}(t-s)^{p-1}s^{q-1}ds=t^{p+q-1}\Gamma(p)\Gamma(q)/\Gamma(p+q)$
\cite[Sec.~1.1]{podlubny},
which with $p=1-\alpha$ and $q=\alpha$ equals $\Gamma(1-\alpha)\Gamma(\alpha)$. Hence
$\partial^{\alpha}_{t}\varphi=\alpha\Gamma(\alpha)/\Gamma(\alpha+1)=1$, so
$\partial^{\alpha}_{t}u=w_{1}$, constant in $t$, and
$\norm{\partial^{\alpha}_{t}u}^{2}_{\LHmone}=T\norm{w_{1}}^{2}_{\HM}<\infty$. Since $u(0)=0$, the
pair representing $u$ in the sense of Definition~\ref{def:space} is $(w_{1},0)$, and no ambiguity
of the kind discussed in Remark~\ref{rem:pairunique} arises. Therefore $u\in\WscaleG$ with
\[
\norm{u}_{W^{\alpha}_{2,2}}
=\frac{T^{\alpha+1/2}}{\sqrt{2\alpha+1}\,\Gamma(\alpha+1)}\norm{w_{1}}_{\HH}
+\sqrt{T}\,\norm{w_{1}}_{\HM}.
\]
The two summands measure different things and neither controls the other: the first is the
spatial regularity of $u$ averaged in time, the second its fractional regularity in time. To see
that the second is not automatic, we take $v(x,t):=h(t)w_{1}(x)$ with $h\in L^{2}(0,T)$
arbitrary: such a $v$ always lies in $\LHone$, but $\partial^{\alpha}_{t}v$ need not exist, still
less belong to $\LHmone$.
\end{example}

Two other explicit families appear in this paper, and neither is an example in the sense just
used. The family $f_{\varrho}$ of Proposition~\ref{prop:nobound} is a concentration
argument: it proves that no constant controls
$\sup_{t}\bigl(\IL{\alpha}\norm{f}^{2}_{\LL}\bigr)(t)$ by the $\LLtwo$ norm of the source, which
is what forces \Hthreep{} into the picture. The source $f(x,t)=F(t)w_{1}(x)$ of
Proposition~\ref{prop:sharp}(ii) is a counterexample: it turns the failure of a bound into the
failure of a conclusion, and it is the content of the optimality statement. Both belong to the
proofs where they occur, and we do not offer them as illustrations.

\section{Conclusions}\label{sec:conclusions}

For the time-fractional diffusion problem \eqref{eq:P} with $\alpha\in(0,1)$, on a bounded
Lipschitz domain and with a bounded, symmetric, uniformly elliptic coefficient matrix, we have
worked under the
structural hypotheses \Hone{} and \Htwo, the initial datum \Hfour, and the two alternatives
\Hthree{} and \Hthreep{} on the source.

The two source hypotheses are ordered, and in one direction only:
$\Hthreep\Rightarrow\Hthree$, since $\alpha<1$ forces every admissible exponent above $2$ and the
interval is bounded, while the converse fails, as $t^{-\alpha/2}w(x)$ shows
(Proposition~\ref{prop:hierarchy}). Hypothesis \Hthree{} is the weaker, \Hthreep{} the stronger,
and we have arranged the paper accordingly.

Under \Hthree{} we obtain a weak solution in $\WscaleG$ for $\alpha\in(\frac12,1)$
(Theorem~\ref{thm:A}), with the time-averaged bounds \eqref{eq:N1} and \eqref{eq:M2A}, the
uniform bound \eqref{eq:LinfA}, and strong convergence of the Galerkin approximations in
$L^{p}\bigl((0,T);\LL\bigr)$ for every $p<2$. Uniqueness holds under \Hthree{} alone and for
every $\alpha\in(0,1)$ (Theorem~\ref{thm:uniqueness}).

Under \Hthreep{} we obtain a solution for every $\alpha\in(0,1)$, lying in $\LinfL$
with the explicit bound \eqref{eq:M1} (Theorem~\ref{thm:B}) and satisfying the stability estimate
\eqref{eq:stability} with the constant \eqref{eq:Keps} written out, and the problem is well posed
in the sense of Hadamard (Corollary~\ref{cor:hadamard}). What the reinforcement buys is therefore
the range $\alpha\le\frac12$ and the explicit constant, and not the uniform estimate itself,
which the mode representation already supplies above $\frac12$ under \Hthree.

Of the three thresholds, two are features of a method. One is the exponent $2/\alpha$, which
Proposition~\ref{prop:spectralbound} lowers to $1/\alpha$ for this operator. The other is the cut
$\alpha>\frac12$, which comes from the window $[1,2\alpha)$ of admissible exponents in the
compactness criterion, and which a linear equation does not need in order to have a solution. The
third is a feature of the problem: for every $\alpha\in(0,\frac12]$ there is a source satisfying
\Hthree{} whose unique weak solution leaves $\LinfL$ (Proposition~\ref{prop:sharp}). The
numerical coincidence of the last two values is not an identity of content.

Existence under \Hthree{} and uniqueness are due to Zacher \cite{zacher09}, the scale
$W^{\alpha}_{2,2}$ to Kubica and Yamamoto \cite{kubica}, and the compactness criterion to Li and
Liu \cite{liliu}; we claim none of these. Three things are ours. The identification of
$[1,2\alpha)$ as the exact window of admissible exponents, with both halves proved
(Proposition~\ref{prop:compactA}), and hence the origin of the restriction in
Theorem~\ref{thm:A}. The identification of $2/\alpha$ as the exact threshold of the energy route,
with its explicit constant and with the concentration argument of
Proposition~\ref{prop:nobound}, which rules out every exponent below it. And the proof that the
gap between the two conclusions is a feature of the problem.

Two open questions, which are the same frontier seen twice. We do not know whether the gap
between $1/\alpha$ and $2/\alpha$ survives without the spectral representation: whether the sharp
bound \eqref{eq:spectralbound} holds for time-dependent coefficients $A=A(x,t)$, in the setting
of \cite{kubica}, or when $A$ is not symmetric, or when the operator carries lower-order terms.
In these cases the energy route of Sections~\ref{sec:caseA} and~\ref{sec:caseB} goes through
essentially unchanged, since we use the symmetry of $a$ only to diagonalise and never in an
energy estimate. The spectral representation on which Section~\ref{sec:sharp} rests is not
available there. The gap between $1/\alpha$ and $2/\alpha$ would then become operative rather
than methodological, and the optimality statement would have to be reached, if at all, by other
means. We do not know either
whether the counterexample of Proposition~\ref{prop:sharp}(ii) extends from $r=2$ to every
$r\le1/\alpha$; the argument we give uses square integrability of the kernel specifically, and we
have not pursued the general case. Finally there is the semilinear problem, where one cannot
dispense with compactness, since weak convergence is not preserved by composition with nonlinear
functions, and where the cut at $\alpha=\frac12$ would therefore become a restriction of the
result and not only of the proof.

\section*{Acknowledgements}

The authors used a large language model for language editing. All mathematical content, proofs
and results are the authors' own.

\end{document}